\documentclass[a4paper]{amsart}
\usepackage{graphicx} 

\newcommand{\C}{{\mathbb{C}}}
\newcommand{\N}{{\mathbb {N}}}
\newcommand{\Z}{\mathbb Z}
\newcommand{\R}{\mathbb R}

\newtheorem{thm}{Theorem}
\newtheorem{defn}{Definition}
\newtheorem{rem}{Remark}
\newtheorem{lem}{Lemma}
\newtheorem{Not}{Notation}
\newtheorem{prop}{Proposition}
\newtheorem{cor}{Corollary}
\newtheorem{Exam}{Example}

\title[On the b-function of annihilating ideals]{On the b-function with respect to weights of annihilating ideals in the Weyl algebra}
\author{Helena Cobo }

\begin{document}

\begin{abstract}
Given a polynomial $f\in\C[x_1,\ldots,x_n]$ and an integer $\ell\in\Z$, we study some properties of the b-function with respect to weights of the annihilating ideal Ann$(f^\ell)$. In some particular cases the expression of the b-function is given explicitly.
\end{abstract}

\maketitle

\section*{Introduction}
\label{Intro}

Many algorithms in algebraic analysis are based on the computation of the so-called b-function with respect to weights of a holonomic ideal (see \cite{SST}). One instance is when computing solutions to holonomic systems (see \cite{OTT}). Another example is the computation of the Bernstein-Sato polynomial. In particular, one way to compute the Bernstein-Sato polynomial passes through the computation of this weighted b-function for a specific weight (see \cite{DA} for an introduction to the subject). The Bernstein-Sato polynomial is an interesting and subtle invariant, very well-known, among other things for its relation with the Monodromy conjecture. Compared with the Bernstein-Sato polynomial, the b-function with respect to weights is much less studied. The case of hypergeometric ideals has attracted some attention due to its rich relation with solving partial differential equations. In \cite{SST} the authors describe the b-function of regular hypergeometric ideals for generic parameters.
 This invariant is studied for irregular hypergeometric systems in \cite{RSW} for specific weights, while in some particular cases it is studied for any weight: in \cite{CC2} for a particular family of hypergeometric systems in the third Weyl algebra and  in \cite{Cobo} for the case of hypergeometric ideals of codimension one. In \cite{C2} a method is proposed to compute the b-function. It does not work in full generality but covers a wide range of examples, and it is used in the present work. To the best of our knowledge this invariant has not been studied apart from the case of hypergeometric systems.

 Apart from the fact that the b-function with respect to weights is used to compute the Bernstein-Sato polynomial, is the b-function with respect to weights in its own subtle and interesting too? Is there any geometry behind this invariant? Answering this questions in one way or another, was the main motivation for this work.

\vspace{2mm}

We focuss on the holonomic ideal ${\rm Ann}(f^\ell)$, with $f\in\C[x_1,\ldots,x_n]$ and $\ell\in\Z$. These annihilating ideals play a role in the computation of the Bernstein-Sato polynomial of $f$ or in the description of the localization $\C[x_1,\ldots,x_n]_f$. These ideals are very difficult to compute in general and hence we only give partial results and compute the b-function in some particular cases.

Inspiring examples were computed using the computer algebra system SINGULAR \cite{singular}.

\vspace{5mm}

{\bf Acknowledgments:} I'm grateful to Francisco Castro Jiménez for introducing me to the world of b-functions with respect to weights.

\section{The weighted b-function of a holonomic ideal}
\label{bf}
We denote by $D_n$ the ring $\C[x_1,\ldots,x_n,\partial_1,\ldots, \partial_n]$ subject to the relations
\[\partial_ix_j=x_j\partial_i+\delta_{ij},\ x_ix_j=x_jx_i,\ \partial_i\partial_j=\partial_j\partial_i\]
for $1\leq i,j\leq n$. It is called the $n$-th Weyl algebra over the field $\C$.

The Weyl algebra is non-commutative and by an ideal in $D_n$ we will always mean a left ideal.

Every operator $P$ in the Weyl algebra has a unique expansion of the form
\[P=\sum_{(\alpha,\beta)\in \N^{2n}}c_{\alpha,\beta}{\mathbf x} 
^\alpha\partial^\beta\]
where only a finite number of coefficients $c_{\alpha,\beta}$ are not zero. It is called the normally ordered expression. We are using the multi-index notation, where ${\mathbf x}^\alpha$ means $x_1^{\alpha_1}\cdots x_n^{\alpha_n}$ and $\partial^\beta$ is $\partial_1^{\beta_1}\cdots\partial_n^{\beta_n}$. For convenience we may write the normally ordered expression as
\[P=\sum_{\beta\in\N^n} p_\beta({\bf x})\partial^\beta\]
where $p_\beta\in\C[x_1,\ldots,x_n]$.

\begin{defn}
A vector $(u,v)=(u_1,\ldots,u_n,v_1,\ldots,v_n)\in\R^{2n}$ is called  a {\em weight vector} (for the Weyl algebra) if
\[u_i+v_i\geq 0,\ \ \mbox{ for }i=1,2\ldots,n.\]
\end{defn}

\begin{defn}
Let $u,v\in\R^n$. For a non-zero operator $P=\sum c_{\alpha,\beta}{\mathbf x}^\alpha\partial^\beta\in D_n$, {\em the initial form} of $P$ with respect to $(u,v)$ is
\[{\rm in}_{(u,v)}(P):=\sum_{\alpha u+\beta v=m}c_{\alpha,\beta}{\mathbf x}^\alpha\partial^\beta\]
where $m={\rm max }\ \{\alpha\cdot u+\beta\cdot v\ |\ c_{\alpha,\beta} \neq 0\}$.

 This definition extends straightforwardly to ideals $I\subseteq D_n$,
 \[{\rm in}_{(u,v)}(I):=D_n\cdot \{{\rm in}_{(u,v)}(P)\ |\ P\in I\}.\]
\end{defn}

\begin{defn}
Given any ideal $I\subseteq D_n$, {\em the characteristic variety} ${\rm Char}(I)$ is the affine variety in $\C^{2n}$ defined by the characteristic ideal ${\rm in}_{({\bf 0},{\bf e})}(I)$, where ${\bf e}=(1,\ldots,1)$.
\end{defn}

\begin{thm}
(\cite{SKK}) Let $I$ be a proper ideal in $D_n$. Every irreducible component of ${\rm Char}(I)$ has dimension at least $n$.
\end{thm}

\begin{defn}
An ideal $I\subseteq D_n$ is said to be {\em holonomic} if ${\rm dim}\big({\rm Char}(I)\big)=n$, i.e., if the dimension of the characteristic variety in $\C^{2n}$ is as small as possible.
\end{defn}

Let $I$ be a left ideal in $D_n$ and $\omega$ a non-zero weight vector in $\R^n$. The initial ideal ${\rm in}_{(-\omega,\omega)}(I)$ is often called the {\em Gröbner deformation} of $I$ with respect to $\omega$. Solutions for Gröbner deformations of $I$ can be extended to solutions for $I$ when $I$ is regular and holonomic (see \cite{SST}). Here, the notion of regular holonomic ideal is a generalization of ordinary differential equations with regular singularities (see Definition 2.4.1 in \cite{SST}).

The b-function of $I$ with respect to $\omega$ is an invariant attached to the Gröbner deformations of $I$. Indeed, for $s:=\omega_1x_1\partial_1+\cdots+\omega_n x_n\partial_n$ we consider the intersection
\[{\rm in}_{(-\omega,\omega)}(I)\cap\C[s],\]
which is an ideal in the principal ideal domain $\C[s]$.

\begin{defn}
The monic generator of ${\rm in}_{(-\omega,\omega)}(I)\cap\C[s]$ is called the b-function of the ideal $I$ with respect to $\omega$. It is denoted by $b_{I,\omega}(s)$.
\end{defn}

In some sense, the b-function of $I$ keeps track of all possible Gröbner deformations of $I$.

\begin{thm}(see \cite{Laurent87})
The $b$--function of a holonomic ideal $I\subseteq D_n$ is not the zero polynomial for any $\omega\in\R^n\setminus\{0\}$.
\end{thm}

\begin{rem}
Given a holonomic ideal $I\subset D_n$ and a weight vector $w\in\R^n\setminus\{\bf 0\}$ it may happen that
\[{\rm in}_{(-\omega,\omega)}(I)=D_n\]
in which case the corresponding b-function is $b_{I,\omega}(s)=1$, but apart from this pathological case, the b-function is a non-zero polynomial of positive degree.

Examples of this behaviour already appeared in \cite{Cobo} for some hypergeometric ideals. A simpler example is the following. Consider the holonomic ideal $I\subseteq D_2$
\[I=D_2\left( y+\partial_x+\partial_y,x+\partial_x+\partial_y+1\right)\]
The Gröbner fan is given by the four quadrants of $\R^2$ together with its faces and we can check that in$_{(-\omega,\omega)}(I)=D_2$ for any $\omega\in\{\lambda(1,0)+\mu(0,1)\ |\ \lambda,\mu>0\}$ the first quadrant.
\label{Entender}
\end{rem}

\vspace{3mm}

From now on we suppose that in$_{(-\omega,\omega)}(I)$ is a proper ideal in $D_n$.

\vspace{3mm}

Given a non-zero polynomial $f\in \C[x]$ we denote by $b_f(s)$ the Bernstein-Sato polynomial associated with $f$. We can recover $b_f(s)$ from $b_{I,\omega}(s)$, for a certain ideal $I$ and weight vector $\omega$, as follows. Consider the Malgrange ideal of $f$, i.e., the left ideal $I_f\subset D_{n+1}$ in the $(n+1)$-th Weyl algebra,  generated by the set $$\left\{x_{n+1}-f(x), \partial_1 + \frac{\partial f}{\partial x_1} \partial_{n+1}, \ldots, \partial_n + \frac{\partial f}{\partial x_n} \partial_{n+1}\right\}.$$
 Then we have
\begin{equation}
b_f(s)=(-1)^{deg\big(b_{I_f,\omega}(s)\big)}b_{I_f,\omega}(-s-1)
\label{special}
\end{equation}
where $w=(0,\ldots,0,1)\in \R^{n+1}$, see  e.g. \cite[Lemma 5.3.11]{SST}.

\vspace{5mm}

There are several algorithms to compute the b-functions of a holonomic ideal with respect to weights. We will follow \cite{Noro}. Once we have, for any $\omega\in\R^n\setminus\{0\}$, the initial ideal ${\rm in}_{(-\omega,\omega)}(I)$, we apply the method of indeterminate coefficients to compute the minimal polynomial of $$s=\omega_1x_1\partial_1+\omega_2x_2\partial_2+\cdots+\omega_nx_n\partial_n$$ as an endomorphism of $D_n/{\rm in}_{(-\omega,\omega)}(I)$.

Therefore we need a Gröbner basis $\mathcal G$ of ${\rm in}_{(-\omega,\omega)}(I)$ with respect to a term order so that we can compute the normal form,
$NF(s^k,\mathcal G)$, for any $k\in \Z_{>0}$.
 Since the condition $$NF(s^j,\mathcal G)+\sum_{k=0}^{j-1}a_kNF(s^k,\mathcal G)=0$$ is equivalent to $$s^j+\sum_{k=0}^{j-1}a_ks^k\in{\rm in}_{(-\omega,\omega)}(I),$$ we can use here any term order on the monomials in the Weyl algebra $D_n$.

We look for the smallest positive integer $n$ such that there exist $a_0,\ldots,a_{n-1}\in\C$ satisfying
\[NF\big(s^n+a_{n-1}s^{n-1}+\cdots+a_0,\mathcal G\big)=0.\]
Then, the b-function of the ideal $I$ with respect to the weight $\omega$ is
\[b_{I,\omega}(s)=s^n+a_{n-1}s^{n-1}+\cdots+a_0.\]

\subsection{The Gröbner fan of an ideal}

Given an ideal $I\subseteq\C[x_1,\ldots,x_n]$, the so-called Gröbner fan of $I$ was defined in \cite{MR}. It is a fan in $\R^n$ describing all possible initial ideals for $I$, i.e., for any element $\tau$ in the Gröbner fan of $I$, the ideal in$_\omega(I)$ is fixed for every $\omega\in\tau$.

\vspace{3mm}

In the non-commutative case, this notion can be extended for an ideal $I\subseteq D_n$ with respect to the filtration  induced by $(u,v)\in\R^{2n}$ such that $u+v\geq 0$ (see \cite{ACG} and \cite{SST}). It is a subdivision of $\{(u,v)\in\R^{2n}\ |\ u+v\geq 0\}$.

Since in this paper we work with the filtration induced by $(-\omega,\omega)$ with $\omega\in\R^n\setminus\{0\}$, we will define the Gröbner fan of an ideal $I\subseteq D_n$ to be the subdivision $\Sigma$ of $\R^n$ such that for any $\tau\in\Sigma$, the ideal
\[{\rm in}_{(-\omega,\omega)}(I)\]
is fixed for any $\omega\in\stackrel{\circ}{\tau}$, where $\stackrel{\circ}{\tau}$ stands for the relative interior of $\tau$, or, in other words, the cones in $\Sigma$ correspond to the different possible initial ideals ${\rm in}_{(-\omega,\omega)}(I)$. It is called the small Gröbner fan of $I$ (see \cite{SST}).

\section{On the b-function with respect to weights of the ideal Ann$(f^\ell)$}
\label{Main}

Given $f\in\C[x_1,\ldots,x_n]$ and $\ell\in\C$ the annihilating ideal of $f^\ell$ is
\[{\rm Ann}(f^\ell)=\{P\in D_n\ |\ P\bullet f^\ell=0\}\]
where the differential operators act here as
\[\begin{array}{l}
\partial_i\bullet f^\ell=\ell f^{\ell-1}\frac{\partial f}{\partial x_i}\\
\\
x_i\bullet f^\ell=x_if^\ell\\
\end{array}\]
for $1\leq i\leq n$.

\vspace{2mm}

Let us fix notation. A generic polynomial $f\in\C[{\bf x}]$ will be denoted as
\[f=\sum_{\gamma\in\N^n}c_\gamma{\bf x}^\gamma\]
where only a finite number of coefficients $c_\gamma$ are non-zero, while a generic operator $P\in D_n$ in its normal form will be denoted as
\[P=\sum_{\alpha,\beta\in\N^n}c_{\alpha\beta}{\bf x}^\alpha\partial^\beta\]
where only a finite number of coefficients $c_{\alpha\beta}$ are non-zero.

\begin{thm} (see \cite{KK})
Let $f\in\C[x_1,\ldots,x_n]$. The ideal Ann$(f^\ell)$ is a regular holonomic $D_n$-ideal for any $\ell\in\C$.
\end{thm}

The ideal Ann$(f^\ell)$ for a given $f$ is very difficult to compute. There are algorithms in \cite{SST} passing through the ring $D_n[s]$, where $s$ is a new variable. Computing the $s$-parametric annihilating ideal Ann$(f^s)$ in $D_n[s]$ and specialize $s=\ell$ works for most of the cases.

\begin{thm} (see Theorem 5.3.13 in \cite{SST})
 Let $\alpha_0$ be the minimal integer root of the Bernstein-Sato polynomial $b_f$ of $f$, and suppose that
 \[\ell\notin\alpha_0+1+\N\]
 Then Ann$(f^\ell)\subseteq D_n$ is obtained from Ann$(f^s)\subseteq D_n[s]$ by replacing $s$ with $\ell$.
\label{thmAnn}
\end{thm}

A possible way to compute the $s$-parametric annihilator Ann$(f^s)$ is given by an algorithm by Briançon y Maisonobe.
See \cite{SST} for other algorithms.

\begin{thm} (\cite{BM})
Consider the {\em shift algebra} $S:=\C\langle\partial_t,\sigma\ |\ \sigma\partial_t=\partial_t\sigma+\partial_t\rangle$ and $D_n^S:=D_n\bigotimes_\C S$. If, for any $f\in\C[x_1,\ldots,x_n]$, we define
\[I:=\langle\sigma+f\cdot\partial_t,\{\partial_i+\frac{\partial f}{\partial x_i}\cdot\partial_t\}_{i=1,\ldots,n}\rangle\subseteq D_n^S\]
then
\[{\rm Ann}(f^s)=I_{|_{\sigma=s}}\cap D_n[s]\]
\label{parametric}
\end{thm}

We do not have a description of the annihilating ideal Ann$(f^\ell)$ for a general $f$ and hence we can not hope to find an explicit description of the b-function.

\vspace{3mm}

For the next definitions and its properties we refer to \cite{Ew}.
\begin{defn}
Given $f=\sum_\gamma c_\gamma{\mathbf x}^\gamma\in\C[x_1,\ldots,x_n]$ the support of $f$ is the set
\[{\rm Supp}(f)=\{\gamma\in\N^n\ |\ c_\gamma\neq 0\}\]
Moreover
\[\begin{array}{l}
{\rm deg}(f)={\rm max}\{|\gamma|\ |\ \gamma\in{\rm Supp}(f)\}\\
\\
{\rm ord}(f)={\rm min}\{|\gamma|\ |\ \gamma\in{\rm Supp}(f)\}\\
\end{array}\]
The {\em Newton polytope} of $f$ is
\[\Gamma(f)={\rm CH}\big({\rm Supp}(f)\big),\]
where CH stands for the convex hull. The vertices of $\Gamma(f)$ are elements of Supp$(f)$.

\vspace{2mm}

The support function of the polytope $\Gamma(f)$ is
\[\begin{array}{cccc}
{\rm ord}_f: &  \R^n & \longrightarrow & \R\\
 & \omega & \mapsto & {\rm min}\{\langle\gamma,\omega\rangle\ |\ \gamma\in\Gamma(f)\}\\
 \end{array}\]
The face of $\Gamma(f)$ determined by $\omega\in\R^n$ is the set
\[\tau_\omega=\{ u\in\Gamma(f)\ |\ \langle u,\omega\rangle={\rm ord}_f(\omega)\}\]
All faces of $\Gamma(f)$ are of this form, and if $\omega=0$ then the face $\tau_\omega$ is the Newton polytope itself.

The dual fan $\Sigma_f$ associated with $\Gamma(f)$ is a fan supported in $\R^n$ formed by the cones
\[C_\tau=\{\omega\in\R^n\ |\ \langle u,\omega\rangle={\rm ord}_f(\omega)\mbox{ for all }u\in\tau\}\]
with $\tau$ running through the faces of $\Gamma(f)$. Notice that for any cone $C$ in $\Sigma_f$, if $\omega,\omega'\in\stackrel{\circ}{C}$ then $\tau_\omega=\tau_{\omega'}$.

 Moreover for any $\tau$ face of $\Gamma(f)$ we set
 \[f_\tau=\sum_{\gamma\in\tau\cap{\rm Supp}(f)}c_\gamma{\bf x}^\gamma\]
 and
 \[V_\tau=\{\omega\in\R^n\ |\ \langle\omega,\gamma\rangle=\langle\omega,\gamma'\rangle\mbox{ for all }\gamma,\gamma'\in\tau\cap{\rm Supp}(f)\}\]
the vector space dual to the face $\tau$. Notice that $C_\tau\subset V_\tau$.
\end{defn}

\begin{lem}
Given $f\in\C[x_1,\ldots,x_n]$ and $\ell\in\N$ there is an injection
\[\begin{array}{cccc}
\varphi_\ell: & \Gamma(f) & \longrightarrow & \Gamma(f^\ell)\\
 & \gamma & \mapsto & \ell\gamma\\
\end{array}\]
which induces a bijection among faces of $\Gamma(f)$ and $\Gamma(f^\ell)$ for any polynomial $f\in\C[x_1,\ldots,x_n]$. As a consequence
\[\Sigma_f=\Sigma_{f^\ell}\]
\label{ell}
\end{lem}

{\em Proof.}
It follows from the fact that the Newton polytope of $f^\ell$ is the Minkowski sum
\[\Gamma(f^\ell)=\Gamma(f)+\stackrel{(\ell)}{\cdots}+\Gamma(f)\]
Hence $\Gamma(f^\ell)$ is an homothecy of $\Gamma(f)$ with center the origin and ratio $\ell$. If $\Gamma_0(f)\subseteq{\rm Supp}(f)$ denotes the 0-dimensional faces of $\Gamma(f)$, then
\[\ell\Gamma_0(f)\subset\Gamma_0(f^\ell)\]
and $\Gamma(f^\ell)$ is determined by $\ell\Gamma_0(f)$. In particular the induced subdivision of $\R^n$ is the same, $\Sigma_{f^\ell}=\Sigma_f$, and the faces of $\Gamma(f)$ and $\Gamma(f^\ell)$ are in bijection by the homothecy. For any $\tau_\ell$ face of $\Gamma(f^\ell)$ there exists a unique $\tau$ face of $\Gamma(f)$ such that $\tau_\ell$ is the homothecy of $\tau$. Moreover if $\{\ell\gamma_1,\ldots,\ell\gamma_k\}$ are the vertices of $\tau_\ell$, then $\{\gamma_1,\ldots,\gamma_k\}$ are the vertices of $\tau$.
\hfill$\Box$

\vspace{2mm}

\begin{defn}
Let $\omega\in\R^n\setminus\{0\}$, a polynomial $f=\sum_\gamma c_\gamma{\bf x}^\gamma\in\C[x_1,\ldots,x_n]$ and an operator $P=\sum_{\alpha,\beta}c_{\alpha\beta}{\bf x}^\alpha\partial^\beta\in D_n$. We call
\[{\rm deg}_\omega(f)={\rm max }\ \{\langle\omega,\gamma\rangle\ |\ \gamma\in{\rm Supp}(f)\}\]
the weighted degree of $f$, and
\[{\rm deg}_{(-\omega,\omega)}(P)={\rm max}\ \{-\langle\omega,\alpha\rangle+\langle\omega,\beta\rangle\ |\ c_{\alpha\beta}\neq 0\}\]
the weighted degree of $P$. Moreover we call
\[{\rm in}_{(-\omega,\omega)}(P)=\sum_{-\langle\omega,\alpha\rangle+\langle\omega,\beta\rangle={\rm deg}_{(-\omega,\omega)}(P)}c_{\alpha\beta}{\bf x}^\alpha\partial^\beta\]
the initial part of $P$ with respect to $(-\omega,\omega)$.
\end{defn}

\vspace{2mm}

\begin{Not}
For any vector $\alpha=(\alpha_1,\ldots,\alpha_n)\in\Z^n$ we will use the notations
\[\begin{array}{l}
|\alpha|=\alpha_1+\cdots+\alpha_n\\
\\
\alpha!=\alpha_1!\cdots\alpha_n!\\
\end{array}\]

We denote by $\{{\bf e_1},\ldots,{\bf e_n}\}$ the canonical bases of $\R^n$, and ${\bf e}=(1,\ldots,1)$.

\label{notations}
\end{Not}

\vspace{2mm}

\begin{lem}
Given a polynomial $f=\sum_\gamma c_\gamma{\bf x}^\gamma\in\C[x_1,\ldots,x_n]$ and a vector $\beta\in\N^n$,
\[\partial^\beta\bullet f=\sum_{\gamma\geq\beta}c_\gamma\frac{\gamma!}{(\gamma-\beta)!}{\bf x}^{\gamma-\beta}\]
where by $\beta\leq\gamma$ we mean $\beta_i\leq\gamma_i$ for $1\leq i\leq n$. Moreover,
\[\partial^\beta\bullet\frac{1}{f}=\sum_{\beta_1+\cdots+\beta_r=\beta}\frac{(-1)^r}{f^{r+1}}\frac{r!}{i_1!\cdots i_s!}\frac{\beta!}{\beta_1!\cdots\beta_r!}\prod_{j=1}^r\big(\partial^{\beta_j}\bullet f\big)\]
where for any partition
\[\beta=\beta_1+\cdots+\beta_r\]
of the vector $\beta$, we denote by $i_1,\ldots,i_s\geq 1$ with $s\leq r$, the positive integers such that $r=i_1+\cdots+i_s$ and we can write this partition as
\[\beta_1+\cdots+\beta_r=\beta_1^{(1)}+\cdots+\beta_{i_1}^{(1)}+\cdots+\beta_1^{(s)}+\cdots+\beta_{i_s}^{(s)}\]
with
\[\beta_1^{(j)}=\beta_2^{(j)}=\cdots=\beta_{i_j}^{(j)}\]
for $1\leq j\leq s$.
\label{Tec}
\end{lem}

{\em Proof.} We provide a proof of this known result, due to a lack of a reference. The first claim of the Lemma is a straightforward consequence of
\[\partial_i\bullet f=\frac{\partial f}{\partial x_i}\]
for $1\leq i\leq n$.

\vspace{3mm}

Let us prove the second claim by (multidimensional) induction on $|\beta|$.

{\bf First step:} $|\beta|=1$. We have $\partial_i\bullet\frac{1}{f}=-\frac{1}{f^2}\partial_i\bullet f$ for $1\leq i\leq n$, which clearly fits with the formula we want to prove, for $\beta=e_i$. The first step of induction follows.

\vspace{2mm}

{\bf Induction hypothesis:} for any $\beta\in\N^n$ with $|\beta|\leq\ell$ the formula holds.

We have to prove it for $\beta'$ such that $|\beta'|=\ell+1$. Any such $\beta'$ is of the form $\beta'=\beta+e_i$ with $1\leq i\leq n$ and $|\beta|=\ell$, so we can apply induction hypothesis on $\beta$.

First we set some notations. For any $1\leq r\leq|\beta|$,
\[P_r(\beta)=\{\beta_1+\cdots+\beta_r\ |\ \beta_1,\ldots,\beta_r\in\N^n,\ \beta_1+\cdots+\beta_r=\beta\}\]
is the set of partitions of the vector $\beta$. Given a partition $p\in P_r(\beta)$, the sequence $i_1,\ldots,i_s$ of integers defined in the statement of this lemma, is denoted by
\[{\rm seq}(p)=(i_1,\ldots,i_s)\]
Moreover by $p!$ we mean $\beta_1!\cdots\beta_r!$. With this notations the induction hypothesis looks
\[\partial^\beta\bullet\frac{1}{f}=\sum_{r=1}^{|\beta|}\sum_{p\in P_r(\beta)}(-1)^r\frac{1}{f^{r+1}}\frac{r!}{{\rm seq}(p)!}\frac{\beta!}{p!}\prod_{j=1}^r\big(\partial^{\beta_j}\bullet f\big)\]

Then, for $1\leq i\leq n$,
\[\begin{array}{ll}
\partial^{\beta+e_i}\bullet\frac{1}{f} & = \displaystyle{\sum_{r=1}^{|\beta|}\sum_{p\in P_r(\beta)}(-1)^r\frac{r!}{{\rm seq}(p)!}\frac{\beta!}{p!}\ \partial^{e_i}\left(\frac{1}{f^{r+1}}\prod_{j=1}^r\big(\partial^{\beta_j}\bullet f\big)\right)}\\
\\
 & =\displaystyle{\sum_{r=1}^{|\beta|}\sum_{p\in P_r(\beta)}\frac{(-1)^{r+1}}{f^{r+2}}\frac{(r+1)!}{{\rm seq}(p)!}\frac{\beta!}{p!}\big(\partial^{e_i}\bullet f\big)\prod_{j=1}^r\big(\partial^{\beta_j}\bullet f\big)}+\\
 \\
 & +\displaystyle{\sum_{r=1}^{|\beta|}\sum_{p\in P_r(\beta)}\frac{(-1)^r}{f^{r+1}}\frac{r!}{{\rm seq}(p)!}\frac{\beta!}{p!}\sum_{j=1}^r\big(\partial^{\beta_1}\bullet f\big)\cdots\big(\partial^{\beta_j+e_i}\bullet f\big)\cdots\big(\partial^{\beta_r}\bullet f\big)}\\
\end{array}\]
Let us define, for any partition $p=\beta_1+\cdots+\beta_r\in P_r(\beta)$, the integer
\[n_i(p)={\rm card}\{1\leq j\leq r\ |\ \beta_j=e_i\}\]
The map
\[P_r(\beta)\ \stackrel{\phi_i}{\longrightarrow} \ P_{r+1}(\beta+e_i)\]
\[p\ \ \ \ \mapsto\ \ \ \ p+e_i\]
is injective and
\[\phi_i(P_r(\beta))=\{p\in P_{r+1}(\beta+e_i)\ |\ n_i(p)>0\}\]
Moreover
\[(p+e_i)!=p!\]
and
\[\begin{array}{ll}
{\rm seq}(p+e_i) & =(n_i(p)+1){\rm seq}(p)!\\
 & =n_i(p+e_i){\rm seq}(p)!\\
 \end{array}\]
Then
\[\begin{array}{ll}
\partial^{\beta+e_i}\bullet\frac{1}{f} & =\displaystyle{\sum_{r=1}^{|\beta|}\sum_{p\in P_{r+1}(\beta+e_i)}\frac{(-1)^{r+1}}{f^{r+2}}\frac{(r+1)!}{{\rm seq}(p)!}\frac{\beta!}{p!}n_i(p)\prod_{j=1}^{r+1}\big(\partial^{\beta_j}\bullet f\big)}+\\
\\
 & +\displaystyle{\sum_{r=1}^{|\beta|}\sum_{p\in P_r(\beta)}\frac{(-1)^r}{f^{r+1}}\frac{r!}{{\rm seq}(p)!}\frac{\beta!}{p!}\sum_{j=1}^r\big(\partial^{\beta_1}\bullet f\big)\cdots\big(\partial^{\beta_j+e_i}\bullet f\big)\cdots\big(\partial^{\beta_r}\bullet f\big)}\\
\\
 & = \displaystyle{\sum_{r=2}^{|\beta|+1}\sum_{p\in P_r(\beta+e_i)}\frac{(-1)^r}{f^{r+1}}\frac{r!}{{\rm seq}(p)!}\frac{\beta!}{p!}n_i(p)\prod_{j=1}^r\big(\partial^{\beta_j}\bullet f\big)}+\\
\\
 & +\displaystyle{\sum_{r=1}^{|\beta|}\sum_{p\in P_r(\beta)}\frac{(-1)^r}{f^{r+1}}\frac{r!}{{\rm seq}(p)!}\frac{\beta!}{p!}\sum_{j=1}^r\big(\partial^{\beta_1}\bullet f\big)\cdots\big(\partial^{\beta_j+e_i}\bullet f\big)\cdots\big(\partial^{\beta_r}\bullet f\big)}\\
 \end{array}\]

In the first sum the term corresponding to $r=|\beta|+1$ is
\begin{equation}
\frac{(-1)^{|\beta|+1}}{f^{|\beta|+2}}(|\beta|+1)!\big(\partial^{e_1}\bullet f\big)^{b_1}\cdots\big(\partial^{e_i}\bullet f\big)^{b_i+1}\cdots\big(\partial^{e_n}\bullet f\big)^{b_n},
\label{eq(beta+1)}
\end{equation}
where $\beta=(b_1,\ldots,b_n)$, since there is only one partition of length $|\beta|+1$, precisely
\[P_{|\beta|+1}(\beta+e_i)=\{e_1+\stackrel{b_1)}{\cdots}+e_1+\cdots+e_i+\stackrel{b_i+1)}{\cdots}+e_i+\cdots+e_n+\stackrel{b_n)}{\cdots}+e_n\}\]
and hence seq$(p)!=(\beta+e_i)!$, $p!=1$ and $n_i(p)=b_i+1$.

Since in the second sum there is no partition of $\beta+e_i$ of length $|\beta|+1$, we conclude that  the contribution of the partition $p=e_1+\stackrel{b_1)}{\cdots}+e_1+\cdots+e_i+\stackrel{b_i+1)}{\cdots}+e_i+\cdots+e_n+\stackrel{b_n)}{\cdots}+e_n$ is given by (\ref{eq(beta+1)}), which fits with the formula we want to prove, for $r=|\beta|+1$.

Moreover notice that in the first sum there are no partitions of length one. Therefore the contribution of $r=1$ is in the second sum, and it is exactly
\[(-1)\frac{1}{f^2}\big(\partial^{\beta+e_i}\bullet f\big),\]
since $P_1(\beta)=\{\beta\}$, seq$(p)!=1$ and $p!=\beta!$. Therefore
\[\begin{array}{ll}
\partial^{\beta+e_i}\bullet\frac{1}{f} & =-\frac{1}{f^2}\partial^{\beta+e_i}\bullet f+\frac{(-1)^{|\beta|+1}}{f^{|\beta|+2}}(|\beta|+1)!\big(\partial^{e_1}\bullet f\big)^{b_1}\cdots\big(\partial^{e_i}\bullet f\big)^{b_i+1}\cdots\big(\partial^{e_n}\bullet f\big)^{b_n}\\
\\
 & +\displaystyle{\sum_{r=2}^{|\beta|}\sum_{p\in P_r(\beta+e_i)}\frac{(-1)^r}{f^{r+1}}\frac{r!}{{\rm seq}(p)!}\frac{\beta!}{p!}n_i(p)\prod_{j=1}^r\big(\partial^{\beta_j}\bullet f\big)}+\\
\\
 & \displaystyle{\sum_{r=2}^{|\beta|}\sum_{p\in P_r(\beta)}\frac{(-1)^r}{f^{r+1}}\frac{r!}{{\rm seq}(p)!}\frac{\beta!}{p!}\sum_{j=1}^r\big(\partial^{\beta_1}\bullet f\big)\cdots\big(\partial^{\beta_j+e_i}\bullet f\big)\cdots\big(\partial^{\beta_r}\bullet f\big)}\\
\end{array}\]

Before continuing with the proof we will write the sums in other way. Let us define the set of vectors
\[V_r(\beta)=\{(\beta_1,\ldots,\beta_r)\in(\N^n)^r\ |\ \beta_1+\cdots+\beta_r=\beta\}\]
Then
\begin{equation}
\sum_{p\in V_r(\beta)}C_p\prod_{j=1}^r\big(\partial^{\beta_j}\bullet f\big)=\sum_{p\in P_r(\beta)}C_p\frac{r!}{{\rm seq}(p)!}\prod_{j=1}^r\big(\partial^{\beta_j}\bullet f\big)
\label{truco}
\end{equation}
whenever $C_p(p\in V_r(\beta))=C_p(p\in P_r(\beta))$, i.e. the coefficient $C_p$ is the same regardless $p\in V_r(\beta)$ or $p\in P_r(\beta)$. This is clearly our case since $C_p=\frac{\beta!}{p!}$. Hence, so far we have the equality
\[\begin{array}{ll}
\partial^{\beta+e_i}\bullet\frac{1}{f} & =-\frac{1}{f^2}\big(\partial^{\beta+e_i}\bullet f\big)+\frac{(-1)^{|\beta|+1}}{f^{|\beta|+2}}(|\beta|+1)!\big(\partial^{e_1}\bullet f\big)^{b_1}\cdots\big(\partial^{e_i}\bullet f\big)^{b_i+1}\cdots\big(\partial^{e_n}\bullet f\big)^{b_n}\\
\\
 & +\displaystyle{\sum_{r=2}^{|\beta|}\sum_{p\in V_r(\beta+e_i)}\frac{(-1)^r}{f^{r+1}}\frac{\beta!}{p!}n_i(p)\prod_{j=1}^r\big(\partial^{\beta_j}\bullet f\big)}+\\
\\
 & +\displaystyle{\sum_{r=2}^{|\beta|}\sum_{p\in V_r(\beta)}\frac{(-1)^r}{f^{r+1}}\frac{\beta!}{p!}\sum_{j=1}^r\big(\partial^{\beta_1}\bullet f\big)\cdots\big(\partial^{\beta_j+e_i}\bullet f\big)\cdots\big(\partial^{\beta_r}\bullet f\big)}\\
\end{array}\]

Notice that $p=(\beta_1,\ldots,\beta_r)\in V_r(\beta+e_i)$ with $n_i(p)<b_i+1$ if and only if there exists $1\leq j\leq r$ such that $\beta_j=v+e_i$ with $v\in\N^n$. Hence in the second sum we are in fact summing over the vectors in $V_r(\beta+e_i)$ with $n_i(p)<b_i+1$, and therefore the partitions in $V_r(\beta+e_i)$ with $n_i(p)=b_i+1$ only appear in the first sum. Then
\[\begin{array}{ll}
\partial^{\beta+e_i}\bullet\frac{1}{f} & =-\frac{1}{f^2}\partial^{\beta+e_i}\bullet f+\frac{(-1)^{|\beta|+1}}{f^{|\beta|+2}}(|\beta|+1)!\big(\partial^{e_1}\bullet f\big)^{b_1}\cdots\big(\partial^{e_i}\bullet f\big)^{b_i+1}\cdots\big(\partial^{e_n}\bullet f\big)^{b_n}\\
\\
 & +\displaystyle{\sum_{r=2}^{|\beta|}\ \sum_{\stackrel{p\in V_r(\beta+e_i)}{n_i(p)=b_i+1}}\frac{(-1)^r}{f^{r+1}}\frac{(\beta+e_i)!}{p!}\prod_{j=1}^r\big(\partial^{\beta_j}\bullet f\big)}+\\
\\
 & +\displaystyle{\sum_{r=2}^{|\beta|}\ \sum_{\stackrel{p\in V_r(\beta+e_i)}{ n_i(p)\leq b_i}}\frac{(-1)^r}{f^{r+1}}\frac{\beta!}{p!}n_i(p)\prod_{j=1}^r\big(\partial^{\beta_j}\bullet f\big)}+\\
\\
 & +\displaystyle{\sum_{r=2}^{|\beta|}\sum_{p\in V_r(\beta)}\frac{(-1)^r}{f^{r+1}}\frac{\beta!}{p!}\sum_{j=1}^r\big(\partial^{\beta_1}\bullet f\big)\cdots\big(\partial^{\beta_j+e_i}\bullet f\big)\cdots\big(\partial^{\beta_r}\bullet f\big)}\\
\end{array}\]

We claim that
\[\begin{array}{c}
 \displaystyle{\sum_{r=2}^{|\beta|}\ \sum_{\stackrel{p\in V_r(\beta+e_i)}{ n_i(p)\leq b_i}}\frac{(-1)^r}{f^{r+1}}\frac{\beta!}{p!}n_i(p)\prod_{j=1}^r\big(\partial^{\beta_j}\bullet f\big)}+\\
  \displaystyle{\sum_{r=2}^{|\beta|}\sum_{p\in V_r(\beta)}\frac{(-1)^r}{f^{r+1}}\frac{\beta!}{p!}\sum_{j=1}^r\big(\partial^{\beta_1}\bullet f\big)\cdots\big(\partial^{\beta_j+e_i}\bullet f\big)\cdots\big(\partial^{\beta_r}\bullet f\big)}=\\
\\
=\displaystyle{\sum_{r=2}^{|\beta|}\sum_{\stackrel{p\in V_r(\beta+e_i)}{n_i(p)\leq b_i}}\frac{(-1)^r}{f^{r+1}}\frac{(\beta+e_i)!}{p!}\prod_{j=1}^r\big(\partial^{\beta_j}\bullet f\big)}\\
\end{array}\]
and, using again (\ref{truco}) we recover the formula that we wanted to prove. Let us the prove the claim. Recall that the second sum
\[S_2=\displaystyle{\sum_{r=2}^{|\beta|}\sum_{p\in V_r(\beta)}\frac{(-1)^r}{f^{r+1}}\frac{\beta!}{p!}\sum_{j=1}^r\big(\partial^{\beta_1}\bullet f\big)\cdots\big(\partial^{\beta_j+e_i}\bullet f\big)\cdots\big(\partial^{\beta_r}\bullet f\big)}\]
runs in fact over the vectors $p\in V_r(\beta+e_i)$ with $n_i(p)\leq b_i$ and there are many repetitions, while in the first sum there are no repetitions. Let us work out the second sum $S_2$. We can write
\[S_2=\displaystyle{\sum_{r=2}^{|\beta|}\sum_{j=1}^r\sum_{\stackrel{p\in V_r(\beta+e_i)}{\beta_j>e_i}}\frac{(-1)^r}{f^{r+1}}\frac{\beta!}{p!}b_i^{(j)}\prod_{k=1}^r\big(\partial^{\beta_k}\bullet f\big)}\]
where we denote $\beta_j=(b_1^{(j)},\ldots,\beta_n^{(j)})$. This follows because, if we fix $j$, the map
\[\Psi_j\ : V_r(\beta)\ \longrightarrow\ V_r(\beta+e_i)\]
\[(\beta_1,\ldots,\beta_r)\ \mapsto\ (\beta_1,\ldots,\beta_j+e_i,\ldots,\beta_r)\]
is injective, which implies that there are no repetitions in the set
\[\{(\beta_1,\ldots,\beta_j+e_i,\ldots,\beta_r)\ |\ (\beta_1,\ldots,\beta_r)\in V_r(\beta)\}\]
Moreover $\Psi_j(V_r(\beta))=\{p\in V_r(\beta+e_i)\ |\ \beta_j>e_i\}$ and for every $p\in V_r(\beta)$ we have $\Psi_j(p)!=(b_i^{(j)}+1)p!$, or in other words, for every $p\in V_r(\beta+e_i)$ with $\beta>e_i$ we have $p!=b_i^{(j)}\Psi_j^{-1}(p)!$.

Now the question is whether there are repetitions in the set
\[\displaystyle{\bigcup_{j=1}^r\{p\in V_r(\beta+e_i)\ |\ \beta_j>e_i\}}\]
The answer is clearly yes since
\[\Psi_{j_1}(V_r(\beta))\cap\Psi_{j_2}(V_r(\beta))\neq\emptyset.\]
It is enough to notice that each $p=(\beta_1,\ldots,\beta_r)\in\cup_{j=1}^r\Psi_j(V_r(\beta))$ appears in exactly every $\Psi_j(V_r(\beta))$ for $1\leq j\leq r$ such that $\beta_j>e_i$. Hence
\[S_2=\displaystyle{\sum_{r=2}^{|\beta|}\sum_{p\in V_r(\beta+e_i)}\frac{(-1)^r}{f^{r+1}}\frac{\beta!}{p!}\sum_{\stackrel{1\leq j\leq r}{\beta_j>e_i}}b_i^{(j)}\prod_{j=1}^r\big(\partial^{\beta_j}\bullet f\big)}\]
Finally the claim follows when we notice that for every $p\in V_r(\beta+e_i)$
\[n_i(p)+\sum_{\stackrel{1\leq j\leq r}{\beta_j>e_i}}b_i^{(j)}=\sum_{\stackrel{1\leq j\leq r}{\beta_j\geq e_i}}b_i^{(j)}=b_i+1\]
\hfill$\Box$

\vspace{3mm}

The goal of the section is to prove the following result.

\begin{thm}
Let $f\in\C[x_1,\ldots,x_n]$ and $\ell\in\Z$. For any $\omega\in\R^n\setminus\{0\}$,
\begin{enumerate}
\item[(i)] $\ell{\rm ord}_f(\omega)$ is a root of $b_{{\rm Ann}(f^\ell),\omega}(s)$

\item[(ii)] If $\ell\in\N$, then
\[b_{{\rm Ann}(f^\ell),\omega}(s)=s-\ell{\rm ord}_f(\omega)\]
\end{enumerate}
\label{uufff}
\end{thm}

In \cite{OTT}  the authors find an upper-bound for the degree of polynomial solutions to the system of differential equations defined by a holonomic ideal $I$, by looking at the smallest integer root of the b-function with respect to weights $\omega$ such that $\omega_i<0$ for $1\leq i\leq n$. Notice that this is not ${\rm ord}_f(\omega)$.

\vspace{3mm}

\begin{rem}
Notice that, if $\ell=0$ and $f\in\C[x_1,\ldots,x_n]$ is any polynomial, then
\[{\rm Ann}(f^\ell)={\rm Ann}(1)=\big(\partial_1,\ldots,\partial_n\big)\]
Then for any $\omega\in\R^n\setminus\{0\}$
\[{\rm in}_{(-\omega,\omega)}\big({\rm Ann}(1)\big)=\big(\partial_1,\ldots,\partial_n\big)\]
Hence $NF\big(\omega_1x_1\partial_1+\cdots+\omega_nx_n\partial_n,\{\partial_1,\ldots,\partial_n\}\big)=0$ and therefore we trivially deduce that
\[b_{{\rm Ann}(1),\omega}(s)=s\]
\label{rem0}
\end{rem}

\vspace{2mm}

\begin{lem}
For every $\gamma\in{\rm Supp}(f)$ there exists an operator $P_\gamma\in D_n$ such that
\[P_\gamma\bullet f=c_\gamma\gamma!\]
\label{lemcte}
\end{lem}

{\em Proof.}
 For any $\gamma\in{\rm Supp}(f)$ we have (see Lemma \ref{Tec})
 \[\partial^{\gamma}\bullet f=c_\gamma\gamma!+\sum_{\gamma<\gamma_1\in{\rm Supp}(f)}c_{\gamma_1}\frac{\gamma_1!}{(\gamma_1-\gamma)!}{\mathbf x}^{\gamma_1-\gamma}\]
If there does not exist $\gamma_1\in{\rm Supp}(f)$ such that $\gamma<\gamma_1$ then we set $P_\gamma=\partial^{\gamma}$. Otherwise
we use the analogous expression for $\partial^{\gamma_1}\bullet f$ and deduce
\[\partial^{\gamma}\bullet f=c_\gamma\gamma!+\sum_{\gamma<\gamma_1}\frac{1}{(\gamma_1-\gamma)!}{\mathbf x}^{\gamma_1-\gamma}\partial^{\gamma_1}\bullet f-\sum_{\gamma<\gamma_1<\gamma_2}\frac{1}{(\gamma_1-\gamma)!}c_{\gamma_2}\frac{\gamma_2!}{(\gamma_2-\gamma_1)!}{\mathbf x}^{\gamma_2-\gamma}\]
Continuing this way, the process is finite, since $f$ is a polynomial, and for those $\gamma$ which are maximal in Supp$(f)$ with respect to $\leq$, we have that $\partial^\gamma(f)=c_\gamma\gamma!$. Hence, denoting $\gamma_0=\gamma$, we find an operator in $D_n$
\[P_\gamma:=\partial^\gamma+\sum_{r=1}^\infty(-1)^{r-1}\sum_{\gamma<\gamma_1<\cdots<\gamma_r}\Big(\prod_{k=1}^r\frac{1}{(\gamma_k-\gamma_{k-1})!}\Big){\mathbf x}^{\gamma_r-\gamma}\partial^{\gamma_r}\]
with the property $P_\gamma\bullet f=c_\gamma\gamma!$. Notice that the sum in $r$ is finite since $\gamma,\gamma_1,\ldots,\gamma_r$ all belong to Supp$(f)$.
\hfill$\Box$

\vspace{3mm}

\begin{rem}
For any $\omega\in\R^n$, the operator $P_\gamma$ just defined is $(-\omega,\omega)$-homogeneous of degree $\langle\gamma,\omega\rangle$.
\label{remhom}
\end{rem}

\begin{lem}
Let $\ell\in\N$ and let $\tau$ be any face of $\Gamma(f)$. For any $1\leq i\leq n$ and any $\omega\in \stackrel{\circ}{C_\tau}$
\[\partial_i^{\ell\lambda_i+1}\in{\rm in}_{(-\omega,\omega)}\big({\rm Ann}(f^\ell)\big)\]
where $\lambda_i={\rm max}\ \{\gamma_i\ |\ \gamma\in\tau\cap{\rm Supp}(f)\}$. Moreover, for any $v\in V_\tau$
\[\sum_{j=1}^nv_jx_j\partial_j-\ell\langle v,\gamma_\tau\rangle\in{\rm in}_{(-\omega,\omega)}\big({\rm Ann}(f^\ell)\big)\]
for any $\gamma_\tau\in\tau$.
\label{lema0}
\end{lem}

{\em Proof.} We first prove the statement for $\ell=1$. The clue is the operator $P_\gamma$ given in Lemma \ref{lemcte} for any $\gamma\in{\rm Supp}(f)$.

\vspace{3mm}

Let us prove that $\partial_i^{\lambda_i+1}$ belongs to the initial ideal ${\rm in}_{(-\omega,\omega)}\big({\rm Ann}(f)\big)$. We claim that there is an operator $Q_i\in{\rm Ann}(f)$ such that ${\rm in}_{(-\omega,\omega)}(Q_i)=\partial_i^{\lambda_i+1}$. Indeed, by Lemma \ref{Tec} we have
\[\partial_i^{\lambda_i+1}\bullet f=\sum_{\gamma\in{\rm Supp}(f),\ \lambda_i+1\leq\gamma_i}c_\gamma\frac{\gamma!}{\big(\gamma-(\lambda_i+1)e_i\big)!}{\mathbf x}^{\gamma-(\lambda_i+1)e_i}\]
We set
\[Q_i:=\partial^{(\lambda_i+1)e_i}-\sum_{\gamma\in{\rm Supp}(f),\ \lambda_i+1\leq\gamma_i}c_\gamma\frac{\gamma!}{\big(\gamma-(\lambda_i+1)e_i\big)!}\frac{1}{c_{\gamma_\tau}\gamma_\tau!}{\mathbf  x}^{\gamma-(\lambda_i+1)e_i}P_{\gamma_\tau}\]
where $\gamma_\tau$ is any element in $\tau\cap{\rm Supp}(f)$. By construction $Q_i\in{\rm Ann}(f)$. The order of any monomial in ${\mathbf x}^{\gamma-(\lambda_i+1)e_i}P_{\gamma_\tau}$ is $-\langle\gamma-(\lambda_i+1)e_i,\omega\rangle+\langle\gamma_\tau,\omega\rangle$ by Remark \ref{remhom}. Moreover
\[-\langle\gamma-(\lambda_i+1)e_i,\omega\rangle+\langle\gamma_\tau,\omega\rangle<\langle(\lambda_i+1)e_i,\omega\rangle\]
because the condition $\lambda_i+1\leq\gamma_i$ implies that for any $\gamma$ appearing in the sum $Q_i-\partial^{(\lambda_i+1)e_i}$, $\gamma$ does not belong to the face $\tau$ and hence $\langle\gamma_\tau,\omega\rangle<\langle\gamma,\omega\rangle$. Therefore ${\rm in}_{(-\omega,\omega)}(Q_i)=\partial_i^{\lambda_i+1}$ as we wanted to prove.

\vspace{3mm}

Finally, if $v\in V_\tau$ we have
\[\Big(\sum_{j=1}^nv_jx_j\partial_j-\langle v,\gamma_\tau\rangle\Big)\bullet f=\sum_{\gamma\notin\tau}c_\gamma\langle v,\gamma-\gamma_\tau\rangle{\mathbf  x}^\gamma\]
We set
\[Q_v:=\sum_{j=1}^nv_jx_j\partial_j-\langle v,\gamma_\tau\rangle-\frac{1}{c_{\gamma_\tau}\gamma_\tau!}\sum_{\gamma\notin\tau}c_\gamma\langle v,\gamma-\gamma_\tau\rangle{\mathbf x}^\gamma P_{\gamma_\tau}\]
where we are considering any $\gamma_\tau\in\tau$. As before it is straightforward to check that $Q_v\in{\rm Ann}(f)$ and that ${\rm in}_{(-\omega,\omega)}(Q_v)=\sum_{j=1}^nv_jx_j\partial_j-\langle v,\gamma_\tau\rangle$.

\vspace{3mm}

Let us suppose now that $\ell>1$. Then we apply the previous result to the polynomial $g=f^\ell$.  By Lemma \ref{ell} the fans $\Sigma_f$ and $\Sigma_{f^\ell}$ are equal and there is a bijection among the faces of $\Gamma(f)$ and $\Gamma(f^\ell)$. Then
\[{\rm max}\ \{\gamma_i\ |\ \gamma\in\tau\cap{\rm Supp}(f^\ell)\}=\ell\lambda_i\]
Here we are using that this maximum is attached by a vertex of $\tau_\ell$ and $\lambda_i$ in a vertex of $\tau$. The first claim is proved.

Regarding the second claim it is enough to point out that the vector spaces $V_\tau$ and $V_{\tau_\ell}$ are the same and that $\gamma_{\tau_\ell}=\ell\gamma_\tau$.
\hfill$\Box$

\vspace{3mm}

\begin{lem}
Let $\ell\in\N$ and $\omega\in\R^n\setminus\{0\}$, then
\[{\rm in}_{(-\omega,\omega)}\big({\rm Ann}(\frac{1}{f^\ell})\big)\subseteq{\rm Ann}\big(\frac{1}{f_\tau^\ell}\big)\]
where $\tau=\tau_\omega$ is the face determined by $\omega$, and
\[f_\tau=\sum_{\gamma\in{\rm Supp}(f)\cap\tau}c_\gamma{\bf x}^\gamma\]
\label{uy}
\end{lem}

{\em Proof.} It is enough to prove it for $\ell=1$, since by Lemma \ref{ell} we have $$(f^\ell)_{\varphi_\ell(\tau)}=(f_\tau)^\ell$$ for any $\tau$ face of $\Gamma(f)$.

\vspace{2mm}

Let us fix $\omega\in\R^n\setminus\{0\}$, notice that, by definition,
\[f_\tau={\rm in}_{-\omega}(f)\]
where $\tau=\tau_\omega$ and for any $u\in\R^n$ ${\rm in}_u(f)=\sum_{\langle u,\gamma\rangle\ {\rm max}}c_\gamma{\bf x}^\gamma$.
We denote $d={\rm deg}_{-\omega}(f)$ the weighted degree of $f$ with respect to $-\omega$, then for every $\gamma\in{\rm Supp}(f)$
\[-\langle\omega,\gamma\rangle\leq d\]
and we have equality if and only if $\gamma\in\tau$.

\vspace{2mm}

For any operator $Q=\sum_{(\alpha,\beta)\in A_Q}c_{\alpha,\beta}{\bf x}^\alpha\partial^\beta\in D_n$ let us study $Q\bullet\frac{1}{f}$. First, for any $\beta\in\N^n$, we have by Lemma \ref{Tec},
\begin{equation}
\partial^\beta\bullet\frac{1}{f}=\sum_{\beta_1+\cdots+\beta_r=\beta}\frac{C_{\beta_1,\ldots,\beta_r}}{f^{r+1}}\prod_{i=1}^r\partial^{\beta_i}\bullet f
\label{eqCOEFF}
\end{equation}
where $C_{\beta_1,\ldots,\beta_r}=(-1)^r\frac{r!}{i_1!\cdots i_s!}\binom{\beta}{\beta_1,\ldots,\beta_r}$ is a non-zero coefficient, and where the sum runs over any $\beta_1,\ldots,\beta_n\in\Z^n_{\geq 0}$ partition of the vector $\beta$. Then
\[\begin{array}{rl}
Q\bullet\frac{1}{f}= & \sum_{(\alpha,\beta)\in A_Q}c_{\alpha\beta}{\bf x}^\alpha\partial^\beta\bullet\frac{1}{f}\\
\\
 = & \sum_{(\alpha,\beta)\in A_Q}c_{\alpha\beta}{\bf x}^\alpha\sum_{\beta_1+\cdots+\beta_r=\beta}\frac{C_{\beta_1,\ldots,\beta_r}}{f^{r+1}}\prod_{i=1}^r\partial^{\beta_i}\bullet f\\
 \\
 = & \frac{1}{f^{M_Q}}\sum_{(\alpha,\beta)\in A_Q}\sum_{\beta_1+\cdots+\beta_r=\beta}C_{\alpha,\beta_1,\ldots,\beta_r}f^{M_Q-r-1}{\bf x}^\alpha\prod_{i=1}^r\partial^{\beta_i}\bullet f\\
 \end{array}\]
where $M_Q={\rm max }\{|\beta| \ |\ (\alpha,\beta)\in A_Q\}+1$ and we have collected the coefficients in $C_{\alpha,\beta_1,\ldots,\beta_r}$ to simplify notation. It follows that
\[Q\bullet\frac{1}{f}=\frac{1}{f^{M_Q}}Q_f\]
where
\[Q_f=\sum_{(\alpha,\beta)\in A_Q}\sum_{\beta_1+\cdots+\beta_r=\beta}C_{\alpha,\beta_1,\ldots,\beta_r}f^{M_Q-r-1}{\bf x}^\alpha\prod_{i=1}^r\big(\partial^{\beta_i}\bullet f\big)\]
is a polynomial in $\C[x_1,\ldots,x_n]$. For any $(\alpha,\beta)\in A_Q$ and any partition $\beta=\beta_1+\cdots+\beta_r$ the monomials of the piece $f^{M_Q-r-1}{\bf x}^\alpha\prod_{i=1}^r\big(\partial^{\beta_i}\bullet f\big)$ of $Q_f$ are the form
\[M:=C_M\big(\prod_{j=1}^{M_Q-r-r}{\bf x}^{\gamma_j}\big){\bf x}^\alpha\prod_{i=1}^r {\bf x}^{\widetilde\gamma_i-\beta_i}\]
for certain non-zero coefficient $C_M$, and with $\gamma_j,\widetilde\gamma_i\in{\rm Supp}(f)$ and $\widetilde\gamma_i\geq\beta_i$. Then
\[\begin{array}{rl}
{\rm deg}_{-\omega}(M) & =-\sum_{j=1}^{M_Q-r-1}\langle\omega,\gamma_j\rangle-\langle\omega,\alpha\rangle-\sum_{i=1}^r\langle\omega,\widetilde\gamma_i-\beta_i\rangle\\
\\
 & =-\langle\omega,\alpha\rangle+\langle\omega,\beta\rangle-\sum_{j=1}^{M_Q-r-1}\langle\omega,\gamma_j\rangle-\sum_{i=1}^r\langle\omega,\widetilde\gamma_i\rangle\\
\\
 & \leq n+(M_Q-r-1)d+rd\\
 \\
 & =n+(M_Q-1)d\\
\end{array}\]
where $n:={\rm deg}_{(-\omega,\omega)}(Q)$. Let us define $P:={\rm in}_{(-\omega,\omega)}(Q)$, then we define $A_P\subseteq A_Q$ such that
\[P=\sum_{(\alpha,\beta)\in A_P}c_{\alpha\beta}{\bf x}^\alpha\partial^\beta\]
Consider now the polynomial
\[H:=\sum_{(\alpha,\beta)\in A_P}\sum_{\beta_1+\cdots+\beta_r=\beta}C_{\alpha,\beta_1,\ldots,\beta_r}f_\tau^{M_Q-r-1}{\bf x}^\alpha\prod_{i=1}^r\big(\partial^{\beta_i}\bullet f_\tau\big)\]
It is straightforward to check that ${\rm in}_{-\omega}(H)=H$ and ${\rm deg}_{-\omega}(H)=n+(M_Q-1)d$.

Notice that $H$ is a piece of $Q_f$ and any monomial in $Q_f$ not appearing in $H$ is of the form
\[\big(\prod_{j=1}^{M_Q-r-1}{\bf x}^{\gamma_j}\big){\bf x}^\alpha\prod_{i=1}^r{\bf x}^{\widetilde\gamma_i-\beta_i}\]
where $(\alpha,\beta)\in A_Q$, $\beta=\beta_1+\cdots+\beta_r$ and at least one of the following conditions hold:
\begin{enumerate}
\item[(i)] $(\alpha,\beta)\in A_Q\setminus A_P$

\item[(ii)] one $\gamma_j\notin\tau$

\item[(iii)] one $\widetilde\gamma_i\notin\tau$
\end{enumerate}
Hence it follows that
\[{\rm in}_{-\omega}(Q_f)=H\]

We claim that
\[P\bullet\frac{1}{f_\tau}=\frac{1}{f_\tau^{M_Q}}H\]
Indeed,
\[\begin{array}{rl}
P\bullet\frac{1}{f_\tau} & =\sum_{(\alpha,\beta)\in A_P}c_{\alpha\beta}{\bf x}^\alpha\partial^\beta\bullet\frac{1}{f_\tau}\\
\\
 & =\sum_{(\alpha,\beta)\in A_P}c_{\alpha\beta}{\bf x}^\alpha\sum_{\beta_1+\cdots+\beta_r=\beta}(-1)^r\frac{r!}{f_\tau^{r+1}}C_{\beta_1,\ldots,\beta_r}\prod_{i=1}^r\partial^{\beta_i}\bullet f_\tau\\
 \\
 & =\frac{1}{f_\tau^{M_P}}\sum_{(\alpha,\beta)\in A_P}\sum_{\beta_1+\cdots+\beta_r=\beta}C_{\alpha,\beta_1,\ldots,\beta_r}f_\tau^{M_P-r-1}{\bf x}^\alpha\prod_{i=1}^r\big(\partial^{\beta_i}\bullet f_\tau\big)\\
 \\
 & =\frac{1}{f_\tau^{M_Q}}\sum_{(\alpha,\beta)\in A_P}\sum_{\beta_1+\cdots+\beta_r=\beta}C_{\alpha,\beta_1,\ldots,\beta_r}f_\tau^{M_Q-r-1}{\bf x}^\alpha\prod_{i=1}^r\big(\partial^{\beta_i}\bullet f_\tau\big)\\
 \end{array}\]
where $M_P$ is defined analogously as $M_Q$, and it follows by definition that $M_P\leq M_Q$.

\vspace{2mm}

To prove the statement, it is enough to prove that for any $P={\rm in}_{(-\omega,\omega)}(Q)\in{\rm in}_{(-\omega,\omega)}\big({\rm Ann}(\frac{1}{f})\big)$ we have that $P\in{\rm Ann}\big(\frac{1}{f_\tau}\big)$.

Since $Q\bullet\frac{1}{f}=\frac{1}{f^{M_Q}}Q_f$, the polynomial $Q_f$ must vanish. In particular the $(-\omega)$-degree part must vanish, i.e., $H=0$ and therefore we deduce that $P\in{\rm Ann}\big(\frac{1}{f_\tau}\big)$ as we wanted to prove.
\hfill$\Box$

\vspace{3mm}

\begin{cor}
The Gröbner fan of the ideal ${\rm Ann}\big(\frac{1}{f^\ell}\big)$ for $\ell\in\N$ is a subdivision of the fan $\Sigma_{f^\ell}$.

Moreover, for any $\omega\in\R^n\setminus\{0\}$ the b-function of ${\rm Ann}\big(\frac{1}{f_{\tau_\omega}^\ell}\big)$ divides the b-function of the ideal ${\rm Ann}\big(\frac{1}{f^\ell}\big)$.
\label{corFAN}
\end{cor}

\begin{rem}
In general the Gröbner fan of Ann$\big(\frac{1}{f^\ell}\big)$ is not equal to $\Sigma_{f^\ell}$, as we can see in the following example. Let $f$ be the polynomial
\[f=y+x^3+y^2\]
The fan $\Sigma_f$ consists of the following cones:
\[\Sigma_f=\{C_1,C_2,C_3,\rho_1,\rho_2,\rho_3,(0,0)\}\]
where
\[\begin{array}{l}
C_1=(1,3)\R_{>0}+(1,0)\R_{>0}\\
C_2=(1,0)\R_{>0}+(-2,-3)\R_{>0}\\
C_3=(-2,-3)\R_{>0}+(1,3)\R_{>0}\\
\rho_1=(1,3)\R_{>0}\\
\rho_2=(1,0)\R_{>0}\\
\rho_3=(-2,-3)\R_{>0}\\
\end{array}\]
We have that $C_2$ is not an element of the Gröbner fan of Ann$\big(\frac{1}{f}\big)$, as we can check with Singular \cite{singular} that if $\omega=(1,-1)\in C_2$, then
\[{\rm in}_{(-\omega,\omega)}\big({\rm Ann}(\frac{1}{f})\big)=D_2\big(\partial_x,y^2\partial_y+2y\big)\]
while for $\omega=(0,-1)\in C_2$ we have
\[{\rm in}_{(-\omega,\omega)}\big({\rm Ann}(\frac{1}{f})\big)=D_2\big(y\partial_x,y^2\partial_y+2y,4x^3\partial_x+6x^2y\partial_y+12x^2-\partial_x\big)\]

\end{rem}

\vspace{5mm}

{\em Proof of Theorem \ref{uufff}.} First we prove (ii). Since the b-function is non-zero, it is enough to prove that $s-\ell {\rm ord}_f(\omega)$ is a multiple of the b-function.

Let $\omega\in\R^n$ be non-zero, there exists a unique cone $C\in\Sigma_f$ such that $\omega\in \stackrel{\circ}{C}$. Let $\tau$ be the face of $\Gamma(f)$ such that $C$ is the cone dual to $\tau$, and let $k$ be the dimension of $\tau$. Then
\[\omega\in C\subset V_\tau\]
and therefore, by Lemma \ref{lema0} we have that
\[\omega_1x_1\partial_1+\cdots+\omega_nx_n\partial_n-\ell\langle\omega,\gamma_\tau\rangle\in{\rm in}_{(-\omega,\omega)}\big({\rm Ann}(f^\ell)\big)\]
which implies (ii).

\vspace{3mm}

Now we prove (i) for $\ell\in\Z_{<0}$.

If $\ell=-1$, notice that for any exponent $\gamma_\tau$ in ${\rm Supp}(f)\cap\tau$,
\[\big(\omega_1x_1\partial_1+\cdots+\omega_nx_n\partial_n+\langle\omega,\gamma_\tau\rangle\big)\bullet\frac{1}{f}=-\frac{1}{f^2}\sum_{\gamma\notin\tau}
c_\gamma\langle\omega,\gamma-\gamma_\tau\rangle{\bf x}^\gamma\]
and we deduce that
\[\omega_1x_1\partial_1+\cdots+\omega_nx_n\partial_n+\langle\omega,\gamma_\tau\rangle\in{\rm Ann}\big(\frac{1}{f_\tau}\big)\]
 Then, since the b-function of Ann$\big(\frac{1}{f_\tau}\big)$ is non-zero, we deduce
\[b_{{\rm Ann}(f_\tau^{-1}),\omega}(s)=s+\langle\omega,\gamma_\tau\rangle\]
By Lemma \ref{uy} it follows that $s+\langle\omega,\gamma_\tau\rangle$ divides $b_{{\rm Ann}(f^{-1}),\omega}(s)$.

When $\ell<-1$ it is enough to notice that, denoting $g=f^{-\ell}$, we have ord$_g(\omega)=-\ell{\rm ord}_f(\omega)$.
\hfill$\Box$

\vspace{3mm}

As it has already been pointed out, the b-function with weights of the ideal ${\rm Ann}(f^\ell)$ depends strongly on the polynomial $f$. Let us see now how we can read from the b-function certain properties of $f$.

\begin{prop}
Let $f\in\C[x_1,\ldots,x_n]$ be a quasi-homogeneous polynomial of degree $d$, with weight vector ${\bf v}\in\R^n\setminus\{\bf 0\}$. Then, for any $\ell\in\Z$
\[b_{{\rm Ann}(f^\ell),\omega}(s)=s-\ell\lambda d\]
for every $\omega=\lambda{\bf v}$ with $\lambda\in\R\setminus\{0\}$.
\label{6enero}
\end{prop}

{\em Proof.} A quasi-homogeneous polynomial $f=\sum_\gamma c_\gamma{\bf x}^\gamma$ is such that $\langle{\bf v},\gamma\rangle=d$. Then, for any $\omega=\lambda{\bf v}$ with $\lambda\in\R\setminus\{0\}$ we have
\[{\rm ord}_f(\omega)={\rm min}\{\langle\gamma,\omega\rangle\ |\ \gamma\in\Gamma(f)\}=\lambda d\]
and if $\ell\in\N$ the result follows by Theorem \ref{uufff}.

Moreover,
\[\begin{array}{rl}
\omega_1x_1\partial_1+\cdots+\omega_nx_n\partial_n\bullet\frac{1}{f^\ell} &=\sum_{i=1}^n\omega_ix_i\partial_i\bullet\frac{1}{f^\ell}\\
 & =\sum_{i=1}^n\omega_ix_i(-\ell)\frac{1}{f^{\ell+1}}\sum_\gamma c_\gamma\partial_i\bullet {\bf x}^\gamma\\
 & =-\frac{\ell}{f^{\ell+1}}\sum_{i=1}^n\omega_i\sum_\gamma c_\gamma\gamma_i{\bf x}^\gamma\\
 & =-\frac{\ell}{f^{\ell+1}}\sum_{i=1}^n\omega_i\gamma_if\\
 & =-\frac{\ell}{f^\ell}\lambda d\\
\end{array}\]
Hence
\[\omega_1x_1\partial_1+\cdots+\omega_nx_n\partial_n+\ell\lambda d\in{\rm Ann}\big(\frac{1}{f^\ell}\big)\]
and the result follows.
\hfill$\Box$

\vspace{3mm}

We deduce that if the b-function $b_{{\rm Ann}(f^\ell),{\bf e}}(s)$ has more than one root then the polynomial $f$ is not homogeneous.
 Recall that ${\bf e}=(1,\ldots,1)\in\R^n$, focussing on the vector ${\bf e}$ we can say more:
\begin{prop}
Let $f\in\C[x_1,\ldots,x_n]$ with ord$(f)=i$. Then $\ell i$ is a root of the b-function $b_{{\rm Ann}(f^\ell),{\bf e}}(s)$ for any $\ell\in\Z$. Moreover, if $i=0$ or $1$, then
\[\begin{array}{rl}
\mbox{ord}(f)=i & \Longleftrightarrow b_{{\rm Ann}(f^{-1}),{\bf e}}(s)=s+i\\
 \end{array}\]
\end{prop}

{\em Proof.} The first claim follows by Theorem \ref{uufff} and the fact that ord$_f({\bf e})={\rm ord}(f)$. 

To prove the equivalence let us suppose first that $i=0$. Notice that for $1\leq i\leq n$ we have
\[P_i:=f\partial_i+\frac{\partial f}{\partial x_i}\in{\rm Ann}(f^{-1})\]
and it is clear that if $f(0)\neq 0$ then ${\rm in}_{(-{\bf e},{\bf e})}(P_i)=f(0)\partial_i$, and hence $\partial_i\in{\rm in}_{(-{\bf e},{\bf e})}\big({\rm Ann}(f^{-1})\big)$. We deduce that $b_{{\rm Ann}(f^{-1}),{\bf e}}(s)=s$.

\vspace{2mm}

The other implication is as follows. Suppose first that $b_{{\rm Ann}(f^{-1}),{\bf e}}(s)=s$. Then
\[x_1\partial_1+\cdots+x_n\partial_n\in{\rm in}_{(-{\bf e},{\bf e})}\big({\rm Ann}(f^{-1})\big)\]
or in other words, there exists an operator $P\in D_n$ such that $P\bullet\frac{1}{f}=0$ and ${\rm in}_{(-{\bf e},{\bf e})}(P)=x_1\partial_1+\cdots+x_n\partial_n$. Then we can write the operator as
\[P=x_1\partial_1+\cdots+x_n\partial_n+\sum_{-|\alpha|+|\beta|<0}c_{\alpha\beta}{\bf x}^\alpha\partial^\beta\]
Recall that we denote $f=\sum_\gamma c_\gamma{\bf x}^\gamma$, and then
\[P\bullet\frac{1}{f}=-\frac{1}{f^2}\sum_\gamma c_\gamma|\gamma|{\bf x}^\gamma+\sum c_{\alpha\beta}{\bf x}^\alpha\partial^\beta\bullet\frac{1}{f}\]
Let $\beta^\ast\in\N^n$ such that there exists $\alpha\in\N$ with $c_{\alpha\beta^\ast}\neq 0$ and with $|\beta|$ maximum. By Lemma \ref{Tec} we can write
\[\begin{array}{rl}
P\bullet\frac{1}{f}= & -\frac{1}{f^{|\beta^\ast|+1}}\left\{f^{|\beta^\ast|-1}\sum_\gamma c_\gamma|\gamma|{\bf x}^\gamma-\right.\\
 & \left.-\sum_{|\alpha|>|\beta|}\sum_{\beta=\beta_1+\cdots+\beta_r}c_{\alpha\beta}{\bf x}^\alpha\prod_{j=1}^r\sum_{\beta_j\leq\gamma\in{\rm Supp}(f)}C_{\gamma,\beta_j}f^{|\beta^\ast|-r}{\bf x}^{\gamma-\beta_j}\right\}\\
 \\
 = & -\frac{1}{f^{|\beta^\ast|+1}}\left\{f^{|\beta^\ast|-1}\sum_\gamma c_\gamma|\gamma|{\bf x}^\gamma-\right.\\
 & \left. -\sum_{|\alpha|>|\beta|}\sum_{\beta=\beta_1+\cdots+\beta_r}\sum_{\beta_j\leq\gamma_j\in{\rm Supp}(f),\ 1\leq j\leq r}C_{\gamma,\beta,r}f^{|\beta^\ast|-r}{\bf x}^{\alpha+\gamma_1-\beta_1+\cdots+\gamma_r-\beta_r}\right\}\\
 \end{array}\]
 for certain coefficients $C_{\gamma,\beta_j}$ and
 $C_{\gamma,\beta,r}$.

Let us suppose that $f_0=0$, then
\[m:={\rm ord}(f)>0\]
Let us choose $\gamma^\ast$ any vector in Supp$(f)$ with the condition $|\gamma^\ast|=m$. Then ${\bf x}^{\gamma^\ast}$ appears in $\sum_\gamma c_\gamma|\gamma|{\bf x}^\gamma$ (notice that this is not the case when $m=0$ and hence $\gamma^\ast={\bf 0}$). Moreover in $f^{|\beta^\ast|-1}\sum c_\gamma|\gamma|{\bf x}^\gamma$ the monomial
\[\big({\bf x}^{\gamma^\ast}\big)^{|\beta^\ast|-1}{\bf x}^{\gamma^\ast}={\bf x}^{|\beta^\ast|\gamma^\ast}\]
appears, and hence it must cancel with a monomial in
\[\sum_{|\alpha|>|\beta|}\sum_{\beta=\beta_1+\cdots+\beta_r}\sum_{\beta_j\leq\gamma_j\in{\rm Supp}(f),\ 1\leq j\leq r}C_{\gamma,\beta,r}f^{|\beta^\ast|-r}{\bf x}^{\alpha+\gamma_1-\beta_1+\cdots+\gamma_r-\beta_r}\]
But these monomials are of the form
\[{\bf x}^a{\bf x}^{\alpha-\beta+\gamma_1+\cdots+\gamma_r}\]
with $a\in{\rm Supp}(f^{|\beta^\ast|-r})$ and $|\alpha|>|\beta|$. Hence
\[\begin{array}{ll}
|a+\alpha-\beta+\gamma_1+\cdots+\gamma_r| & =|a|+|\alpha|-|\beta|+|\gamma_1|+\cdots+|\gamma_r|\\
 & >|a|+|\gamma_1|+\cdots+|\gamma_r|\\
 & \geq m(|\beta^\ast|-r)+mr\\
 & =m|\beta^\ast|\\
 \end{array}\]
Hence the monomial ${\bf x}^{|\beta^\ast|\gamma^\ast}$, which has degree $m|\beta^\ast|$ can not cancel and this contradicts the fact that $P\in{\rm Ann}(\frac{1}{f})$.

\vspace{3mm}

If ord$(f)=1$ then
\[\mathcal C:=\{1\leq i\leq n\ |\ {\bf e_i}\in{\rm Supp}(f)\}\neq\emptyset\]
(recall that by $\{{\bf e_1},\ldots,{\bf e_n}\}$ we denote the canonical bases of $\R^n$). We claim that
\begin{enumerate}
\item[(i)] If $j\notin\mathcal C$, then $\partial_j\in{\rm in}_{(-{\bf e},{\bf e})}\big({\rm Ann}(f^{-1})\big)$

\item[(ii)] $\sum_{i\in\mathcal C}x_i\partial_i+1\in{\rm in}_{(-{\bf e},{\bf e})}\big({\rm Ann}(f^{-1})\big)$
\end{enumerate}
Therefore it is clear that $x_1\partial_1+\cdots+x_n\partial_n\equiv -1$ modulo ${\rm in}_{(-{\bf e},{\bf e})}\big({\rm Ann}(f^{-1})\big)$, which implies that the b-function is $b_{{\rm Ann}(f^{-1}),{\bf e}}(s)=s+1$.

\vspace{2mm}

Let us prove the claims. To prove the first claim, recall that the operator
\[\begin{array}{rl}
P_{ij} & :=\frac{\partial f}{\partial x_i}\partial_j-\frac{\partial f}{\partial x_j}\partial_i\in{\rm Ann}(f^{-1})\\
\\
& =\big(\sum_\gamma c_\gamma\gamma_i{\bf x}^{\gamma-{\bf e}_i}\big)\partial_j-\big(\sum_\gamma c_\gamma\gamma_j{\bf x}^{\gamma-{\bf e}_j}\big)\partial_i\\
\end{array}\]
where the first sum runs over $\gamma\in{\rm Supp}(f)$ such that $\gamma\geq {\bf e_i}$ and the second such that $\gamma\geq {\bf e_j}$. Notice that if we pick $j\notin\mathcal C$ and $i\in\mathcal C$, in the second sum we have $\gamma\neq {\bf e_j}$ because $j\notin\mathcal C$. Hence we deduce that ${\rm in}_{(-{\bf e},{\bf e})}\big(P_{ij}\big)=c_{{\bf e_i}}\partial_j$, since for any $\gamma\in{\rm Supp}(f)\setminus\{{\bf e_i}\}$ with $\gamma>{\bf e_i}$ we have $-\langle\gamma-{\bf e_i},{\bf e}\rangle+1<1$.

Regarding the second claim consider the operator
\[Q_i:=f\partial_i+\frac{\partial f}{\partial x_i}\in{\rm Ann}(f^{-1})\]
with $i\in\mathcal C$. Since for any $\gamma\in {\rm Supp}(f)\setminus\{{\bf e_i}\ |\ i\in\mathcal C\}$ we have $\langle\gamma,{\bf e}\rangle\geq 2$ (here we use that $f(0)=0$), we deduce that
\[{\rm in}_{(-{\bf e},{\bf e})}\big(Q_i\big)=\sum_{j\in\mathcal C}c_{{\bf e_j}}x_j\partial_i+c_{{\bf e_i}}\in{\rm in}_{(-{\bf e},{\bf e})}\big({\rm Ann}(f^{-1})\big)\]
If $|\mathcal C|=1$ we are done. Otherwise consider the operator $P_{ij}$ with $i,j\in\mathcal C$. Then, as before we have
\[{\rm in}_{(-{\bf e},{\bf e})}(P_{ij})=c_{{\bf e_i}}\partial_j-c_{{\bf e_j}}\partial_i\in{\rm in}_{(-{\bf e},{\bf e})}\big({\rm Ann}(f^{-1})\big)\]
And this, together with ${\rm in}_{(-{\bf e},{\bf e})}(Q_i)$ proves the second claim.

\vspace{3mm}

The other implication goes as follows. If $b_{{\rm Ann}(f^{-1}),{\bf e}}(s)=s+1$, then
\[x_1\partial_1+\cdots+x_n\partial_n+1\in{\rm in}_{(-{\bf 1},{\bf 1})}\big({\rm Ann}(f^{-1})\big)\]
By Lemma \ref{uy},
\[x_1\partial_1+\cdots+x_n\partial_n+1\in{\rm Ann}(f_\tau^{-1})\]
where $\tau=\tau_{{\bf 1}}$ is the face associated to the vector ${\bf e}$. Notice that since $f_\tau=\sum_\gamma c_\gamma{\bf x}^\gamma$ where the sum runs over $\gamma\in{\rm Supp}(f)$ such that $|\gamma|={\rm ord}(f)$, we have
\[\big(x_1\partial_1+\cdots+x_n\partial_n+1\big)\bullet\frac{1}{f_\tau}=\frac{-1}{f_\tau}\big({\rm ord}(f)-1\big)\]
and this vanishes if and only if ord$(f)=1$.
\hfill$\Box$

\vspace{3mm}

\begin{rem}
Notice that ord$(f)=i>1$ is not equivalent with $b_{{\rm Ann}(f^{-1}),{\bf e}}(s)=s+i$, as can be checked with the polynomial $f=x^3+y^4$. We can compute with Singular \cite{singular} that
\[{\rm Ann}(f^{-1})=D_2\langle3x^2\partial_y-4y^3\partial_x,4x\partial_x+3y\partial_y+12\rangle\]
and
\[b_{{\rm Ann}(f^{-1}),{\bf e}}(s)=\left(s+3\right)\left(s+\frac{10}{3}\right)\left(s+\frac{11}{3}\right)\]
\end{rem}

\begin{lem}
If $f$ is a homogeneous polynomial, then
\[b_{{\rm Ann}(\frac{1}{f}),{\bf e}}(s)=s+d\mbox{ if and only if }f\mbox{ has degree }d\]
\label{HomE}
\end{lem}

\vspace{3mm}

Unfortunately we can not describe the b-function of Ann$(f^\ell)$ for $\ell\in\Z_{<0}$ in general. There are however some special very simple cases where we can give the b-function or at least an upper bound, i.e., a multiple.

\vspace{3mm}

\begin{lem}
 Let $f=(a_0+a_1x_1+\cdots+a_nx_n)^k$ with $k\in\N_{>0}$, we have, for any $\ell\in\Z$,
\[b_{{\rm Ann}(f^\ell),\omega}(s)=s-\ell{\rm ord}_f(\omega)\]
\label{linear}
\end{lem}

{\em Proof.} Since  ${\rm ord}_{f^k}(\omega)=k{\rm ord}_f(\omega)$ it is enough to prove it for $k=1$.

Notice that for $1\leq j\leq n$
\begin{equation}
\sum_{i=1}^nx_i\partial_i+\frac{a_0}{a_j}\partial_j-\ell\in{\rm Ann}(f^\ell)
\label{eqast1}
\end{equation}
and for $1\leq i<j\leq n$
\begin{equation}
a_i\partial_j-a_j\partial_i\in{\rm Ann}(f^\ell)
\label{eqast2}
\end{equation}
Suppose that $\omega_j\geq 0$ for $1\leq j\leq n$ and $a_0\neq 0$, then ${\rm ord}_f(\omega)=0$ and we are done since by (\ref{eqast1}), we have that $\partial_j\in{\rm in}_{(-\omega,\omega)}\big({\rm Ann}(f^\ell)\big)$ for $1\leq j\leq n$.

Otherwise, let us define the set
\[I_\omega=\{1\leq i\leq n\ |\ \omega_i={\rm ord}_f(\omega)\}\]
We have that $I_\omega\neq\emptyset$. Notice that either there exists $\omega_j<0$ or $a_0=0$, and in both cases we have
\begin{equation}
\sum_{i=1}^nx_i\partial_i-\ell\in{\rm in}_{(-\omega,\omega)}\big({\rm Ann}(f^\ell)\big)
\label{eqast3}
\end{equation}
By (\ref{eqast2}) we deduce that if $i\notin I_\omega$ then $\partial_i\in{\rm in}_{(-\omega,\omega)}\big({\rm Ann}(f^\ell)\big)$. Then, by (\ref{eqast3}) we have that $\sum_{i\in I_\omega}x_i\partial_i-\ell\in{\rm in}_{(-\omega,\omega)}\big({\rm Ann}(f^\ell)\big)$ and hence
\[\begin{array}{ll}
s & =\omega_1x_1\partial_1+\cdots+\omega_nx_n\partial_n\\
 & \equiv \sum_{i\in I_\omega}\omega_ix_i\partial_i\mbox{ mod }{\rm in}_{(-\omega,\omega)}\big({\rm Ann}(f^\ell)\big)\\
 & = {\rm ord}_f(\omega)\sum_{i\in I_\omega}x_i\partial_i\\
 & \equiv \ell{\rm ord}_f(\omega)\\
 \end{array}\]
\hfill$\Box$

\vspace{3mm}

\begin{lem}
Let  $f=(x_1-a_1)^{\alpha_1}(x_2-a_2)^{\alpha_2}\cdots(x_n-a_n)^{\alpha_n}$, with $\alpha=(\alpha_1,\ldots,\alpha_n)\in\N^n$. Then, for any $\ell\in\Z$
\[b_{{\rm Ann}(f^\ell),\omega}(s)=s-\ell{\rm ord}_f(\omega)\]
\label{NC}
\end{lem}

{\em Proof.}
It is straightforward to check that the ideal Ann$(f^\ell)$ is generated by
\[\left\{\begin{array}{cl}
(x_i-a_i)\partial_i-\ell\alpha_i & \mbox{ if }\alpha_i\neq 0\\
\partial_i & \mbox{ if }\alpha_i=0\\
\end{array}\right.\]
for $1\leq i\leq n$. Let $\omega\in\R^n\setminus\{\bf 0\}$, then we deduce that for any $\mathcal G$ Gröbner basis of in$_{(-\omega,\omega)}\big({\rm Ann}(f^\ell)\big)$, we have that if $\alpha_i=0$ then $\partial_i\in\mathcal G$, and if $\alpha_i\neq 0$ then:
\[\begin{array}{cl}
\partial_i\in\mathcal G &\mbox{ if }a_i\neq 0\mbox{ and }\omega_i>0\\
x_i\partial_i-\ell\alpha_i\in\mathcal G & \mbox{ if }a_i=0\mbox{ and }\omega_i>0\\
x_i\partial_i-\ell\alpha_i\in\mathcal G & \mbox{ if }\omega_i<0\\
(x_i-a_i)\partial_i-\ell\alpha_i\in\mathcal G & \mbox{ if }\omega_i=0\\
\end{array}\]
Therefore
\[\begin{array}{rl}
NF(s,\mathcal G) & =NF(\omega_1x_1\partial_1+\cdots+\omega_nx_n\partial_n)\\
 & =\sum_{\omega_i>0}NF(\omega_ix_i\partial_i,\mathcal G)+\sum_{\omega_i<0}NF(\omega_ix_i\partial_i,\mathcal G)\\
 & =\sum_{\omega_i>0,\ a_i=0}\ell\alpha_i\omega_i+\sum_{\omega_i<0}\ell\alpha_i\omega_i\\
 \end{array}\]
To finish notice that the monomials of $f$ are of the form
 \[C_{j_1,\ldots,j_n}x_1^{j_1}\cdots x_n^{j_n}\]
 for certain coefficient $C_{j_1,\ldots,j_n}$, and where
 \[\left\{\begin{array}{cl}
 0\leq j_i\leq \alpha_i & \mbox{ if }a_i\neq 0\\
 \\
 j_i=\alpha_i & \mbox{ if }a_i=0\\
 \end{array}\right.\]
 Then
 \[\begin{array}{rl}
 {\rm ord}_f(\omega) & :={\rm min}\{\langle\omega,\alpha\rangle\ |\ \alpha\in{\rm Supp}(f)\}\\
  & =\langle\omega,{\bf j}\rangle\\
  \end{array}\]
  where ${\bf j}\in\Z^n$ is defined by
  \[j_i=\left\{\begin{array}{cl}
  0 & \mbox{ if }\omega_i\geq 0\mbox{ and }a_i\neq 0\\
  \\
  \alpha_i & \mbox{ otherwise }\\
  \end{array}\right.\]
  and the result follows.
\hfill$\Box$

\vspace{2mm}

Can we characterize when the b-function of Ann$(f^{-1})$ has global degree one? That is, when $b_{{\rm Ann}(f^{-1}),\omega}(s)=s-{\rm ord}_f(\omega)$ for any $\omega\in\R^n\setminus\{\bf 0\}$. It is not a property of divisors of linear differential type, as we can check in the following example.

\begin{Exam}
Consider the function $f=xy(x+y)$ defining a hyperplane arrangement. We have
\[{\rm Ann}(f^{-1})=D_2\left(x\partial_x+y\partial_y+3,(xy+y^2)\partial_y+x+2y\right)\]
and $b_{{\rm Ann}(f^{-1}),(1,-1)}(s)=(s-1)(s+1)$.
\end{Exam}

\subsubsection{The case when $f\in\C[x,y]$ is a homogeneous polynomial}

We deal in this section with a polynomial
\[f=\alpha_0x^d+\alpha_1x^{d-1}y+\cdots+\alpha_dy^d\]
homogeneous of degree $d$. We denote
\[R_d=\{(i,j)\in\N^2\ |\ i+j=d\}\]
Then Supp$(f)\subseteq R_d$.

The commutative ring $\C[\theta]=\C[\theta_1,\theta_2]$ where $\theta_i=x_i\partial_i$ can be embed in $D_2$ via the formula
\begin{equation}
\partial_i^mx_i^m=(\theta_i+1)(\theta_i+2)\cdots(\theta_i+m)
\label{eqT}
\end{equation}

\begin{lem}
Let $k\in\{1,2\}$ and $j\leq i,d\in\N$. Consider the ideal $$I=\C[\theta_1,\theta_2]\langle\theta_1+\theta_2+d,(\theta_k+j)(\theta_k+j+1)\cdots(\theta_k+i)\rangle$$
 The ideal $D_2I$ is holonomic and
\[b_{D_2I,\omega}(s)=\left(\prod_{k=0}^{i-j}\big(s+\omega_1(d-i+k)+\omega_2(i-k)\big)\right)_{red}\]
for every $\omega\in\R^2\setminus\{\bf 0\}$.
\label{techEnero}
\end{lem}

{\em Proof.} The ideal $D_2I$ is holonomic, since in$_{(0,{\bf e})}(D_2I)=D_2\langle x_1\partial_1+x_2\partial_2,x_k^{i-j+1}\partial_k^{i-j+1}\rangle$ and hence Char$(I)$ has all components of dimension 2.

In order to apply now some results of \cite{Cobo}, we need the concept of torus-fixed ideal in $D_n$, which has not been defined in the text. We use Lemma 2.3.1 in \cite{SST} as an alternative definition: an ideal in $D_n$ is torus-fixed if and only if it is generated by elements of the form ${\bf x}^ap({\bf x}\partial)\partial^b$ for certain $a,b\in\N^n$ and a polynomial $p$. As a consequence, for any ideal $J\subseteq\C[\theta]$, the ideal $D_nJ$ is torus-fixed. And in particular the ideal $D_2I$ we are dealing with is a torus-fixed ideal.

We can assume that $k=2$ to simplify the notation. The case $k=1$ is completely analogous. Let us consider the transformation $T:\C[\theta_1,\theta_2]\longrightarrow\C[\theta_1,\theta_2]$ given by
\[\left\{\begin{array}{l}
\theta_1\ \mapsto\ \theta_1-\theta_2-d+i\\
\theta_2\ \mapsto\ \theta_2-i\\
\end{array}\right.\]
Then $T(I)=\C[\theta_1,\theta_2]\langle\theta_1,\theta_2(\theta_2-1)\cdots(\theta_2-i+j+1)\rangle$ and by Theorem 3.30 in \cite{Cobo} the result follows, since $T(\omega_1\theta_1+\omega_2\theta_2)=\omega_1\theta_1+(\omega_2-\omega_1)\theta_2-\omega_1d-(\omega_2-\omega_1)i$ and $R={\rm std}(\C[\theta_1,\theta_2]\langle\theta_1,\theta_2^{i-j+1}\rangle)=\{(0,0),(0,1),\ldots,(0,i-j)\}$.
\hfill$\Box$

\vspace{2mm}

\begin{prop}
Let $f\in\C[x,y]$ be a homogeneous polynomial of degree $d$. Then, for any $\omega\in\R^2\setminus\{\bf 0\}$, the b-function $b_{{\rm Ann}(\frac{1}{f}),\omega}(s)$ is a divisor of
\[\left(\prod_{i=0}^d\big(s+i\omega_1+(d-i)\omega_2\big)\right)_{red}\]
\label{PropEnero}
\end{prop}

{\em Proof.} We denote
\[B(s)=\left(\prod_{i=0}^d\big(s+i\omega_1+(d-i)\omega_2\big)\right)_{red}\]
The cases $f=\alpha x^iy^{d-i}$ and $f=(ax+by)^k$ have already been treated in Lemma \ref{NC} and Lemma \ref{linear}.

If $\omega_1=\omega_2$, then as a consequence of Lemma \ref{HomE} we have that $b_{{\rm Ann}(\frac{1}{f}),\omega}(s)=s+\omega_1d$ which coincides with $B(s)$.

The idea of the proof when $\omega_1\neq\omega_2$ relies on the following facts:

\begin{enumerate}
\item[i)] For any holonomic ideal $I\subseteq D_n$ and any $\omega \in\R^n\setminus\{\bf 0\}$
\[{\rm in}_{(-\omega,\omega)}(I)\cap\C[\omega_1x_1\partial_1+\cdots+\omega_nx_n\partial_n]=\left({\rm in}_{(-\omega,\omega)}(I)\cap\C[\theta]\right)\cap\C[\omega_1\theta_1+\cdots+\omega_n\theta_n]\]

\item[ii)] For any two holonomic ideal $I\subseteq J\subseteq D_n$, and any $\omega\in\R^n\setminus\{\bf 0\}$, we have that the b-function $b_{J,\omega}(s)$ divides $b_{I,\omega}(s)$.

\end{enumerate}
Hence we have to prove that
\[C[\theta]\langle\theta_1+\theta_2+d,\theta_k(\theta_k+1)\cdots(\theta_k+d)\rangle\subseteq{\rm in}_{(-\omega,\omega)}\big({\rm Ann}(\frac{1}{f})\big)\cap\C[\theta]\]
for $k\in\{1,2\}$, and then the result follows by Lemma \ref{techEnero}.

For any $\omega\in\R^2\setminus\{\bf 0\}$ it is clear that $\theta_1+\theta_2+d\in{\rm in}_{(-\omega,\omega)}\big({\rm Ann}(\frac{1}{f})\big)\cap\C[\theta_1,\theta_2]$. The rest of the proof is devoted to prove that the operator $\theta_k(\theta_k+1)\cdots(\theta_k+d)$ belongs to the ideal ${\rm in}_{(-\omega,\omega)}\big({\rm Ann}(\frac{1}{f})\big)\cap\C[\theta_1,\theta_2]$ for $k\in\{1,2\}$. For this we introduce the following notation:
\[P_x=f\partial_x+\frac{\partial f}{\partial_x}\mbox{ and }P_y=f\partial_y+\frac{\partial f}{\partial_y}\]
We have that $P_x,P_y\in{\rm Ann}\big(\frac{1}{f}\big)$ and we write
\[P_x=x^{a_1}y^{b_1}P_{0,x}\mbox{ and }P_y=x^{a_2}y^{b_2}P_{0,y}\]
We have that
\[P_{0,x},P_{0,y}\in{\rm Ann}\big(\frac{1}{f}\big)\]
and, since $f=\alpha_0x^d+a_1x^{d-1}y+\cdots+a_dy^d$ we have
\[P_{0,x}=\sum_i\alpha_i\left(x^{d-a_1-i}y^{i-b_1}\partial_x+(d-i)x^{d-a_1-i-1}y^{i-b_1}\right)\]
and
\[P_{0,y}=\sum_i\alpha_i\left(x^{d-a_2-i}y^{i-b_2}\partial_y+ix^{d-a_2-i-}y^{i-b_2-1}\right)\]
For $\omega\in\R^2\setminus\{\bf 0\}$ with $\omega_1\neq\omega_2$, we have
\[{\rm in}_{(-\omega,\omega)}(P_{0,x})=\alpha_{i^*}\left(x^{d-a_1-i^*}y^{i^*-b_1}\partial_x+(d-i^*)x^{d-a_1-i^*-1}y^{i^*-b_1}\right)\]
and
\[{\rm in}_{(-\omega,\omega)}(P_{0,y})=\alpha_{i^*}\left(x^{d-a_2-i^*}y^{i^*-b_2}+i^*x^{d-a_2-i^*}y^{i^*-b_2-1}\right)\]
where
\[i^*=\left\{\begin{array}{ll}
{\rm max}\{0\leq i\leq d\ |\ \alpha_i\neq 0\} & \mbox{ if }\omega_1>\omega_2\\
\\
{\rm min}\{0\leq i\leq d\ |\ \alpha_i\neq 0\} & \mbox{ if }\omega_1<\omega_2\\
\end{array}\right.\]
Notice that if $\omega_1<\omega_2$ then $i^*<d$ while if $\omega_1>\omega_2$ then $i^*>0$, because otherwise $f$ would be a monomial.

\vspace{2mm}

We distinguish cases:

\vspace{2mm}

$\bullet$ If $\alpha_0=0$. Then $b_1>0$. We have proved that
\[Q=x^{d-a_1-i^*}y^{i^*-b_1}\partial_x+(d-i^*)x^{d-a_1-i^*-1}y^{i^*-b_1}\in{\rm in}_{(-\omega,\omega)}\big({\rm Ann}(\frac{1}{f})\big)\]
Hence $\bar Q=\partial_x^{d-a_1-i^*-1}\partial_y^{i^*-b_1}Q\in{\rm in}_{(-\omega,\omega)}\big({\rm Ann}(\frac{1}{f})\big)\cap\C[\theta]$. Indeed, by equation (\ref{eqT}) it follows that
\[\bar Q=(\theta_1+1)(\theta_1+2)\cdots(\theta_1+d-a_1-i^*-1)(\theta_1+d-i^*)(\theta_2+1)\cdots(\theta_2+i^*-b_1)\]
Eliminating $\theta_2$ in $\bar Q$ using the Euler operator $\theta_1+\theta_2+d$ we deduce that the operator
\[(\theta_1+1)(\theta_1+2)\cdots(\theta_1+d-a_1-i^*-1)(\theta_1+d-i^*)(\theta_1+d+b_1-i^*)\cdots(\theta_1+d-1)\]
belongs to the ideal $\C[\theta]\langle\theta_1+\theta_2+d,\bar Q\rangle\subseteq{\rm in}_{(-\omega,\omega)}\big({\rm Ann}(\frac{1}{f})\big)\cap\C[\theta]$. Since $a_1\geq 0$ and $b_1>0$ we finally deduce that
\[\C[\theta]\langle\theta_1+\theta_2+d,\theta_1(\theta_1+1)\cdots(\theta_1+d)\rangle\subseteq{\rm in}_{(-\omega,\omega)}\big({\rm Ann}(\frac{1}{f})\big)\cap\C[\theta]\]

\vspace{2mm}

$\bullet$ If $\alpha_d=0$. Then $a_2>0$ and we have to mimic the proof of the previous case to the operator
\[Q=x^{d-i^*-a_2}y^{i^*-b_2}\partial_y+i^*x^{d-i^*-a_2}y^{i^*-b_2-1}\in{\rm in}_{(-\omega,\omega)}\big({\rm Ann}(\frac{1}{f})\big)\]
to deduce that
\[\C[\theta]\langle\theta_1+\theta_2+d,\theta_1(\theta_1+1)\cdots(\theta_1+d)\rangle\subseteq{\rm in}_{(-\omega,\omega)}\big({\rm Ann}(\frac{1}{f})\big)\cap\C[\theta]\]

\vspace{2mm}

$\bullet$ If $\alpha_0\neq 0$ and $\alpha_d\neq 0$. Then $a_1=b_1=a_2=b_2=0$ and
\[i^*=\left\{\begin{array}{cl}
d & \mbox{ if }\omega_1>\omega_2\\
0 & \mbox{ if }\omega_1<\omega_2\\
\end{array}\right.\]
We have proved that $x_k^d\partial_k+dx_k^{d-1}\in{\rm in}_{(-\omega,\omega)}\big({\rm Ann}(\frac{1}{f})\big)$, where
\[k=\left\{\begin{array}{cl}
2 & \mbox{ if }\omega_1>\omega_2\\
1 & \mbox{ if }\omega_1<\omega_2\\
\end{array}\right.\]
Then
\[\bar Q=\partial_k^{d-1}\big(x_k^d\partial_k+dx_k^{d-1}\big)\in{\rm in}_{(-\omega,\omega)}\big({\rm Ann}(\frac{1}{f})\big)\cap\C[\theta]\]
and since $\bar Q=\partial_k^{d-1}x_k^{d-1}\big(x_k\partial_k+d\big)=(\theta_k+1)\cdots(\theta_k+d)$, we deduce that
\[\theta_k(\theta_k+1)\cdots(\theta_k+d)\in{\rm in}_{(-\omega,\omega)}\big({\rm Ann}(\frac{1}{f})\big)\cap\C[\theta]\]
as we wanted to prove.
\hfill$\Box$

\vspace{3mm}

\begin{rem}
Notice that if we denote $R_d=\{(i,j)\in\N^2\ |\ i+j=d\}$, then we can write the polynomial $B(s)$ as
\[B(s)=\prod_{{\bf k}\in R_d}\big(s+\langle\omega,{\bf k}\rangle\big)\]
On the other hand we have that Supp$(f)\subseteq R_d$. However the polynomial $\prod_{{\bf k}\in{{\rm Supp}(f)}}\big(s+\langle\omega,{\bf k}\rangle\big)$ is neither a divisor nor a multiple of the b-function, as the following example shows. Let $f=x^3-x^2y+2y^3$ and $\omega=(2,1)$, then
\[b_{{\rm Ann}(\frac{1}{f}),\omega}(s)=(s+3)(s+4)(s+5)\]
while
\[\prod_{{\bf k}\in{{\rm Supp}(f)}}\big(s+\langle\omega,{\bf k}\rangle\big)=(s+6)(s+5)(s+3)\]
\end{rem}

\vspace{3mm}

\begin{rem}
When the polynomial $f\in\C[x_1,\ldots,x_n]$ with $n>2$ the result is no longer true. Consider $f=x^3-x^2y+2y^3+2z^2x$ and $\omega=(2,-1,0)$, then we can check with Singular that
\[b_{{\rm Ann}(\frac{1}{f}),\omega}(s)=(s-3)(s-2)(s-\frac{4}{3})(s+\frac{1}{3})(s+1)\]
The reason is that, arguing as before, we can prove that there is an operator $\bar Q=(\theta_k+j)(\theta_k+j+1)\cdots(\theta_k+i)\in{\rm in}_{(-\omega,\omega)}({\rm Ann}(\frac{1}{f})\cap\C[\theta]$, for certain $j\leq i\in\N$, and $1\leq k\leq n$, but this, together with the Euler operator $x_1\partial_1+\cdots+x_n\partial_n+d$ do not gives rise to a holonomic ideal contained in ${\rm in}_{(-\omega,\omega)}({\rm Ann}(\frac{1}{f}))\cap\C[\theta]$.
\end{rem}

\subsubsection{The b-function of annihilating ideals of the curves $x^p+y^q$}

Now we consider the easiest case of $f$ quasi-homogeneous, $f=x^p+y^q$ with $p\leq q$.

\vspace{2mm}

In this section we consider the family of curves with equation $f=x^p+y^q$. Without loss of generality we can suppose that $p\leq q$.

In this particular case we are able to describe the annihilating ideal, by using Theorem \ref{parametric}. Eliminating the variable $\partial_t$ is the ideal $I$ we obtain that the parametric ideal is
\[{\rm Ann}(f^s)=D_2[s]\cdot\big(px^{p-1}\partial_y-qy^{q-1}\partial_x, qx\partial_x+py\partial_y-pqs\big)\]
Since for this family of cusps it is well known that the smallest integer root of the Bernstein-Sato polynomial is $\alpha_0=-1$, we deduce by Theorem \ref{thmAnn} that for any $\ell\in\C\setminus\N$
\begin{equation}
{\rm Ann}\big(f^\ell\big)=D_2\cdot\big(px^{p-1}\partial_y-qy^{q-1}\partial_x,qx\partial_x+py\partial_y-\ell pq\big)
\label{AnnI}
\end{equation}

If $\ell\in\N$, the b-function of the annihilating ideal of a polynomial is described in Theorem \ref{uufff}.

\begin{prop}
The weighted b-function of the ideal Ann$(f^\ell)$ where $\ell\in\N$ and $f=x^p+y^q$, is
\[b_{\ell,\omega}(s)=\left\{
\begin{array}{cl}
s-p\ell\omega_1 & \mbox{ if }p\omega_1\leq q\omega_2\\
\\
s-q\ell\omega_2 & \mbox{ otherwise}\\
\end{array}\right.\]
\label{fl}
\end{prop}

From now on $\ell\in\Z_{<0}$, and we devote the rest of the section to compute the b-function of the annihilating ideal of the rational function $f^\ell$. The description is much more involved than the case $\ell>0$. The annihilating ideal of the rational function $f^\ell$ is given in (\ref{AnnI}).

\begin{Not}
We denote the generators of this ideal by $$H=px^{p-1}\partial_y-qy^{q-1}\partial_x$$ and the Euler operator $$E_\ell=qx\partial_x+py\partial_y-\ell pq.$$  Let us introduce the following notation
\[\begin{array}{l}
Q_{1,\ell}=y^q\partial_y-\ell qy^{q-1}\ \mbox{ and }\ P_{1,\ell}=Q_{1,\ell}+x^p\partial_y\\
\\
Q_{2,\ell}=x^p\partial_x-\ell px^{p-1}\ \mbox{ and }\ P_{2,\ell}=Q_{2,\ell}+y^q\partial_x\\
\end{array}\]
\end{Not}

\begin{lem}
The initial ideal of ${\rm Ann}\big(f^\ell\big)$ with respect to the weight vector $\omega=(\omega_1,\omega_2)\in\R^2\setminus\{0\}$ and $\ell\notin\N$ is
\[{\rm in}_{(-\omega,\omega)}\big({\rm Ann}(f^\ell)\big)=\left\{\begin{array}{cl}
\big(H,E_\ell,P_{1,\ell}\big) & \mbox{ if }p\omega_1=q\omega_2\\
\\
\big(y^{q-1}\partial_x,E_\ell,y^q\partial_y-\ell qy^{q-1}\big) & \mbox{ if }p\omega_1>q\omega_2\\
\\
\big(x^{p-1}\partial_y,E_\ell,x^p\partial_x-\ell px^{p-1}\big) & \mbox{ if }p\omega_1<q\omega_2\\
\end{array}\right.\]
\label{initialcusp}
\end{lem}

The proof of the Lemma is postponed to the next section, where we will prove a slightly more general result (see Lemma \ref{inCuspQH}).

\begin{defn}
For $\omega\in\R^2\setminus\{0\}$ we define the set $\mathcal G_{\ell,\omega}$ as the set of generators of the initial ideal of Ann$(f^\ell)$. Then, with the above notations,
\[\mathcal G_{\ell,\omega}=\left\{\begin{array}{cl}
\{H,E_\ell,P_{1,\ell}\} & \mbox{ if }p\omega_1=q\omega_2\\
\\
\{y^{q-1}\partial_x,E_\ell,Q_{1,\ell}\} & \mbox{ if }p\omega_1>q\omega_2\\
\\
\{x^{p-1}\partial_y,E_\ell,Q_{2,\ell}\} & \mbox{ if }p\omega_1<q\omega_2\\
\end{array}\right.\]
\end{defn}

\begin{lem}
For any $\omega\in\R^2\setminus\{0\}$ and any $\ell\in\C\setminus\N$ there exists a term order $<^\omega$ in $D_2$ such that
\begin{enumerate}
\item[(i)] If $p\omega_1>q\omega_2$
\[LT(E_\ell)=qx\partial_x\]
and
\[LT(Q_{1,\ell})=y^q\partial_y\]
\item[(ii)] If $p\omega_1<q\omega_2$
\[LT(E_\ell)=py\partial_y\]
and
\[LT(Q_{2,\ell})=x^p\partial_x\]
\item[(iii)] If $p\omega_1=q\omega_2$
\[LT(H)=-qy^{q-1}\partial_x\]
\[LT(E_\ell)=qx\partial_x\]
and
\[LT(P_{1,\ell})=y^q\partial_y\]
\end{enumerate}
\label{termOrder}
\end{lem}

{\em Proof.}
To show that the conditions on the leading terms are compatible, we only have to find $(u,v)\in\R^4_{\geq 0}$ representing $<^\omega$ (i.e. such that $<_{(u,v)}$ has the properties of the statement). To have the conditions  of (i) we simply need that
$u_1+v_1>u_2+v_2$ and $qu_2+v_2>(q-1)u_2$, i.e., any $(u,v)\in\R^4_{\geq 0}$ such that
\[u_1+v_1>u_2+v_2>0\]
works.

Analogously, for the case (ii) we look for $(u,v)\in\R^4_{\geq 0}$ such that
\[u_2+v_2>u_1+v_1>0\]
Finally, in (iii) we do not have so much freedom, since there are four conditions,
\[\begin{array}{c}
(q-1)u_2+v_1>(p-1)u_1+v_2\\
u_1+v_1>u_2+v_2\\
qu_2+v_2>pu_1+v_2\\
qu_2+v_2>(q-1)u_2\\
\end{array}\]
or equivalently
\[\begin{array}{c}
u_1+v_1>u_2+v_2>0\\
qu_2>pu_1\\
\end{array}\]
which obviously has a positive solution.
\hfill$\Box$

\vspace{3mm}

\begin{prop}
For any $\omega\in\R^2\setminus\{0\}$ the set $\mathcal G_{\ell,\omega}$ is a Gröbner basis of ${\rm in}_{(-\omega,\omega)}\big({\rm Ann}(f^\ell)\big)$ with respect to the term order $<^\omega$.
\label{GBcusp}
\end{prop}

{\em Proof.}
Suppose first that $p\omega_1>q\omega_2$. Then, with respect to the term order $<^\omega$, we have
\[\begin{array}{l}
LT(E_\ell)=qx\partial_x\\
\\
LT(Q_{1,\ell})=y^q\partial_y\\
\end{array}\]
We start computing $S$-polynomials.
\[\begin{array}{ll}
S(y^{q-1}\partial_x,E_\ell) & =qx y^{q-1}\partial_x-y^{q-1} E_\ell\\
 & =-py^q\partial_y+\ell pqy^{q-1}=-pQ_{1,\ell}\longrightarrow_{\mathcal G_\omega}\{0\}\\
 \end{array}\]

 \[\begin{array}{ll}
 S(y^{q-1}\partial_x,Q_{1,\ell}) & =y\partial_y (y^{q-1}\partial_x)-\partial_xQ_{1,\ell}\\
  & =(q-1+\ell q)y^{q-1}\partial_x\longrightarrow_{\mathcal G_\omega}\{0\}\\
  \end{array}\]

\[\begin{array}{ll}
S(E_\ell,Q_{1,\ell}) & = y^q\partial_y(E_\ell)-qx\partial_x Q_{1,\ell}\\
 & =(p-\ell pq)y^q\partial_y+py^{q+1}\partial_y^2+\ell q^2xy^{q-1}\partial_x\\
 &=-p(q-1)Q_{1,\ell}+py\partial_yQ_{1,\ell}+\ell q^2xy^{q-1}\partial_x\longrightarrow_{\mathcal G_\omega}\{0\}\\
 \end{array}\]

\vspace{2mm}

Suppose now that $p\omega_1<q\omega_2$, then, with respect to the term order $<^\omega$ we have
\[\begin{array}{l}
LT(E_\ell)=py\partial_y\\
\\
LT(Q_{2,\ell})=x^p\partial_x\\
\end{array}\]
We start computing $S$-polynomials.
\[\begin{array}{ll}
S(x^{p-1}\partial_y,E_\ell) & =pyx^{p-1}\partial_y-x^{p-1}E_\ell\\
 & =-qx^p\partial_x+\ell pqx^{p-1}=-qQ_{2,\ell}\longrightarrow_{\mathcal G_\omega}\{0\}\\
 \end{array}\]

\[\begin{array}{ll}
S(x^{p-1}\partial_y,Q_{2,\ell}) & =x\partial_x(x^{p-1}\partial_y)-\partial_y(Q_{2,\ell})\\
 & =(p-1+\ell p)x^{p-1}\partial_y\longrightarrow_{\mathcal G_\omega}\{0\}\\
 \end{array}\]

\[\begin{array}{ll}
S(E_\ell,Q_{2,\ell}) & =x^p\partial_x(E_\ell)-py\partial_y(Q_{2,\ell})\\
 & =(q-\ell pq)x^p\partial_x+qx^{p+1}\partial_x^2+\ell p^2x^{p-1}y\partial_y\\
 & =qx\partial_x Q_{2,\ell}+\ell p^2x^{p-1}y\partial_y-q(p-1)Q_{2,\ell}\longrightarrow_{\mathcal G_\omega}\{0\}\\
 \end{array}\]

\vspace{2mm}

Finally, if $p\omega_1=q\omega_2$, we have

\[\begin{array}{ll}
S(H,E_\ell) & =xH+y^{q-1}E_\ell\\
 & =px^p\partial_y+py^q\partial_y-\ell pqy^{q-1}=pP_{1,\ell}\longrightarrow_{\mathcal G_\omega}\{0\}\\
 \end{array}\]

\[\begin{array}{ll}
S(H,P_{1,\ell}) & =y\partial_y H+q\partial_x P_{1,\ell}\\
 & =px^{p-1}y\partial_y^2-\big(\ell q^2+q(q-1)\big)y^{q-1}\partial_x+pqx^{p-1}\partial_y+qx^p\partial_x\partial_y\\
 & =(\ell q+q-1)H+x^{p-1}\partial_yE_\ell\longrightarrow_{\mathcal G_\omega}\{0\}\\
\end{array}\]

\[\begin{array}{ll}
S(E_\ell,P_{1,\ell}) & =y^q \partial_yE_\ell-qx\partial_xP_{1,\ell}\\
 & =(-\ell pq+p)y^q\partial_y+py^{q+1}\partial_y^2-pqx^p\partial_y-qx^{p+1}\partial_x\partial_y+\ell q^2 xy^{q-1}\partial_x\\
 & =\ell qy^{q-1}E_\ell-x^p\partial_yE_\ell+(-\ell pq-pq+p)P_{1,\ell}+py\partial_y P_{1,\ell}\longrightarrow_{\mathcal G_\omega}\{0\}\\
\end{array}\]
\hfill$\Box$

\vspace{3mm}

\begin{thm}
The b-function with respect to weight vector $\omega\in\R^2\setminus\{0\}$,  of the ideal ${\rm Ann}\big(f^\ell\big)$ with $\ell\in\Z_{<0}$ and $f=x^p+y^q$, is given by

\[b_{{\rm Ann}(\frac{1}{f^\ell}),\omega}(s)=\left\{\begin{array}{cl}
(s-\ell q\omega_2)\prod_{i=1}^{q-1}\big(s-(\ell p+\frac{p}{q}i)\omega_1+i\omega_2\big) & \mbox{ if }p\omega_1>q\omega_2\\
\\
(s-\ell p\omega_1)\prod_{i=1}^{p-1}\big(s+i\omega_1-(\ell q+\frac{q}{p}i)\omega_2\big) & \mbox{ if }p\omega_1<q\omega_2\\
\\
s-\ell p\omega_1 & \mbox{ if } p\omega_1=q\omega_2\\
\end{array}\right.\]
In particular, notice that there are not multiple roots.
\label{bpq}
\end{thm}

{\em Proof.}
  Let $\omega\in\R^2\setminus\{\bf 0\}$. When $p\omega_1=q\omega_2$, we reduce $s=\omega_1x\partial_x+\omega_2y\partial_y$ with respect to the Euler operator $E_\ell$ and get $NF(s,\mathcal G_{\ell,\omega})=\ell p\omega_1$. The result follows in this case.

 If $p\omega_1>q\omega_2$, let us denote by $I$ the initial ideal of Ann$(f^\ell)$ with respect to $\omega$, that is
 \[I=D_2\langle y^{q-1}\partial_x,E_\ell,y^q\partial_y-\ell qy^{q-1}\rangle\]
 The commutative ideal $I\cap\C[\theta_1,\theta_2]$ is generated by $P_1,E_\ell$ and $P_2$ where
 \[\begin{array}{l}
 P_1=\partial_y^{q-1}y^{q-1}x\partial_x\\

 \ \ \ \ =\theta_1(\theta_2+1)\cdots(\theta_2+q-1)\\

 \\

 P_2=\partial_y^{q-1}(y^q\partial_y-\ell qy^{q-1})\\

 \ \ \ \ =(\theta_2+1)\cdots(\theta_2+q-1)(\theta_2-\ell q)\\
 \end{array}\]

 With respect to a term order $<$ such that $LT(E_\ell)=qx\partial_x$, we have that
 \[NF(s,\mathcal G)=\left(\omega_2-\frac{p}{q}\omega_1\right)y\partial_y+\ell p\omega_1\]
where $\mathcal G$ is a Gröbner basis of $I$ with respect to the term order $<$. Then a power $s^k$ can not be reduced by $\mathcal G$ till $k=q$, when we reduce by $P_2$. Therefore deg$(b_{\ell,\omega}(s))\geq q$.

Denoting $\alpha_2=\omega_2-\frac{p}{q}\omega_1$ and $\alpha_0=\ell p\omega_1$ (notice that $\alpha_2\neq 0$), we deduce
\[\theta_2=\frac{NF(s,\mathcal G)-\alpha_0}{\alpha_2}\]
Since $P_2\in I$ we deduce that
\[NF\left(\big(\frac{s-\alpha_0}{\alpha_2}+1\big)\big(\frac{s-\alpha_0}{\alpha_2}+2\big)\cdots\big(\frac{s-\alpha_0}{\alpha_2}+q-1\big)
\big(\frac{s-\alpha_0}{\alpha_2}-\ell q\big),\mathcal G\right)=0\]
Multiplying by $\alpha_2^q$ we get the monic polynomial $b_{\ell,\omega}(s)$.

 The proof in the case $p\omega_1<q\omega_2$ goes analogously.\hfill$\Box$

\subsubsection{Deforming the curve $x^p+y^q$}
Due to the specialization (\ref{special}) to the Bernstein-Sato polynomial, it is reasonable to expect that the weighted b-function is not a topological invariant. It is easy to check that this is indeed the case.

Before trying to compute the b-function of deformations of the polynomial $x^p+y^q$, notice that our computation for Ann$\big(f^\ell\big)$ where $f$ is the cusp, relies strongly on the existence of an Euler operator in the annihilating ideal. This is also the case for quasi-homogeneous deformations of the curve $x^p+y^q$.

\begin{lem}
Let $p$ and $q$ be two positive integers, $d={\rm gcd}(p,q)$, $p_1=\frac{p}{d}$ and $q_1=\frac{q}{d}$, and let us consider the quasi-homogeneous polynomial
\[f=x^p+y^q+\sum_{j=1}^{d-1}\lambda_jx^{p-jp_1}y^{jq_1}\]
with $\lambda_j\in\C$. Then, for $\ell\in\Z_{<0}$,
\[{\rm Ann}\big(f^\ell\big)=\Big(q_1x\partial_x+p_1y\partial_y-\ell\frac{pq}{d}, H\Big)\]
where $H$ is the following operator
\[H=\Big(qy^{q-1}+\sum_{j=1}^{d-1}\lambda_jjq_1x^{p-jp_1}y^{jq_1-1}\Big)\partial_x-\Big(px^{p-1}+\sum_{j=1}^{d-1}\lambda_j(p-jp_1)x^{p-jp_1-1}y^{jq_1}\Big)\partial_y\]
\label{qh}
\end{lem}

{\em Proof.} It is consequence of Theorem \ref{parametric}.\hfill$\Box$

\vspace{2mm}

\begin{rem}
Notice that if $p$ and $q$ are coprime, there are no quasi-homogeneous deformations of the cusp.
\end{rem}

As always we denote the Euler operator by $E_\ell$. Let us compute the Gröbner deformations.

\begin{Not}

We define, for any integer $\ell>0$ the operators
\[\begin{array}{l}
P_{1,\ell}=f\partial_y-\ell\big(qy^{q-1}+\sum_{j=1}^{d-1}\lambda_jjq_1x^{p-jp_1}y^{jq_1-1}\big)\\
P_{2,\ell}=f\partial_x-\ell\big(px^{p-1}+\sum_{j=1}^{d-1}\lambda_j(p-jp_1)x^{p-jp_1-1}y^{jq_1}\big)\\
\end{array}\]
\end{Not}

\begin{lem}
Let $f=x^p+y^q+\sum_{j=1}^{d-1}\lambda_jx^{p-jp_1}y^{jq_1}$ be a quasi-homogeneous polynomial (with the notations of Lemma \ref{qh}). Then, for any $\ell\in\C\setminus\N$ and any $\omega\in\R^2\setminus\{0\}$ we have
\[{\rm in}_{(-\omega,\omega)}\big({\rm Ann}\big(f^\ell\big)\big)=\left\{\begin{array}{cl}
\big(H,E_\ell,P_{1,\ell}\big) & \mbox{ if }p\omega_1=q\omega_2\\
\\
\big(y^{q-1}\partial_x,E_\ell,y^q\partial_y-\ell qy^{q-1}\big) & \mbox{ if }p\omega_1>q\omega_2\\
\\
\big(x^{p-1}\partial_y,E_\ell,x^p\partial_x-\ell px^{p-1}\big) & \mbox{ if }p\omega_1<q\omega_2\\
\end{array}\right.\]
 Therefore, if we set $f_0=x^p+y^q$ and $\omega\in\R^2\setminus\{0\}$ is such that $p\omega_1\neq q\omega_2$,
\[{\rm in}_{(-\omega,\omega)}\big({\rm Ann}(f^\ell)\big)={\rm in}_{(-\omega,\omega)}\big({\rm Ann}(f_0^\ell)\big)\]
\label{inCuspQH}
\end{lem}

{\em Proof.}
To describe ${\rm in}_{(-\omega,\omega)}\big({\rm Ann}(f^\ell)\big)$ we need a Gröbner basis of Ann$(f^\ell)$ with respect to $<_{(-\omega,\omega)}$. Since $<_{(-\omega,\omega)}$ is only a monomial order, we will use the homogenized order of $<_{(-\omega,\omega)}$ to compute a Gröbner basis of the ideal generated by $\bar H$ and $\bar E_\ell$ in $D_2^{(h)}$ with respect to the homogenized order $<_{(-\omega,\omega)}^h$, and use Theorem 1.2.4 in \cite{SST}.

\vspace{2mm}

The homogenized operators are (recall that $p\leq q$)
\[\begin{array}{l}
\bar H=\Big(qy^{q-1}+\sum_{j=1}^{d-1}\lambda_jjq_1h^{(q-p)(1-\frac{j}{d})}x^{p-jp_1}y^{jq_1-1}\Big)\partial_x\\
\ \ \ \ \ \ -\Big(ph^{q-p}x^{p-1}+\sum_{j=1}^{d-1}\lambda_j(p-jp_1)
h^{(q-p)(1-\frac{j}{d})}x^{p-jp_1-1}y^{jq_1}\Big)\partial_y\\
\\
\bar E_\ell=q_1x\partial_x+p_1y\partial_y-\ell\frac{pq}{d}h^2\\
\\
\bar P_{1,\ell}=h^{q-p}x^p\partial_y+y^q\partial_y+\sum_{j=1}^{d-1}\lambda_jh^{(q-p)(1-\frac{j}{d})}x^{p-jp_1}y^{jq_1}\partial_y\\
\ \ \ \ \ \ \ \ -\ell\Big(qh^2y^{q-1}+\sum_{j=1}^{d-1}\lambda_jjq_1h^{(q-p)(1-\frac{j}{d})+2}x^{p-jp_1}y^{jq_1-1}\Big)\\
\\
\bar P_{2,\ell}=h^{q-p}x^p\partial_x+y^q\partial_x+\sum_{j=1}^{d-1}\lambda_jh^{(q-p)(1-\frac{j}{d})}x^{p-jp_1}y^{jq_1}\partial_x\\
\ \ \ \ \ \ \ \ -\ell\Big(ph^{q-p+2}x^{p-1}+\sum_{j=1}^{d-1}\lambda_j(p-jp_1)h^{(q-p)(1-\frac{j}{d})+2}x^{p-jp_1-1}y^{jq_1}\Big)\\
\end{array}\]

\vspace{3mm}

Now we distinguish the cases $p\omega_1<q\omega_2$, $p\omega_1>q\omega_2$ and $p\omega_1=q\omega_2$ and proceed with the Buchberger algorithm. We leave the computations of the $S-$polynomials to an Appendix.\hfill$\Box$

\vspace{3mm}

As a consequence we deduce that the weighted b-function is not sensitive to quasi-homogeneous deformations of the cusp.

\begin{prop}
Let $f=x^p+y^q+\sum_{j=1}^{d-1}\lambda_ix^{p-jp_1}y^{jq_1}$ be a quasi-homogeneous polynomial (with the notations of Lemma \ref{qh}), and let $f_0=x^p+y^q$. Then, for any $\omega\in\R^2\setminus\{0\}$ and any $\ell\in\Z$,
\[b_{{\rm Ann}(f^\ell),\omega}(s)=b_{{\rm Ann}(f_0^\ell),\omega}(s)\]
\label{quasidef}
\end{prop}

{\em Proof.}
First we deal with the case $\ell\notin\N$.
Once the two initial ideals for Ann$\big(f^\ell\big)$ and Ann$\big(f_0^\ell\big)$ coincide when $p\omega_1\neq q\omega_2$, the set $\mathcal G_\omega$ is the same and we can use Lemma \ref{termOrder} and Proposition \ref{GBcusp}. Then Theorem \ref{bpq} holds for $f$. If $p\omega_1=q\omega_2$ the proof of Theorem \ref{bpq} uses only the Euler operator $E_\ell$ and hence it holds in this case too.

\vspace{2mm}

When $\ell\in\N$, it follows by Theorem \ref{uufff}, since it is easy to check that
\[{\rm ord}_{f_0}(\omega)={\rm ord}_f(\omega)\]
for any $\omega\in\R^2$.
\hfill$\Box$

\vspace{10mm}

However this is not the case for another deformations of $f$, leaving fixed the Newton polygon. The main problem is that now the annihilating ideal is difficult to compute, and we do not have a system of generators of it. Moreover we can easily see there is no Euler operator in the annihilating ideal.

We consider next deformations of the cusp, with the same Newton polygon, and we will see that, the b-functions coincide for some $\omega$, even though the initial ideals are different.

\begin{rem}($\mu$-constant deformations of the cusp)

Inspired by \cite{CN} let us consider the $\mu$-constant deformations of $x^p+y^q$ with gcd$(p,q)=1$. They are of the form 
\[x^p+y^q+\sum\lambda_jx^{a_j}y^{b_j}\]
where $j$ runs over the set
\[\{j\ |\ j+1=a_jq+b_jp-pq,\ 0\leq a_j\leq p-2, \ 0\leq b_j\leq q-2,\ \frac{a_j}{p}+\frac{b_j}{q}>1\}\]
Let $f_0=x^p+y^q$ and $f$ a $\mu$-constant deformation of $f_0$, then in \cite{CN} it is proved that $b_f(s)\neq b_{f_0}(s)$, but they share some roots, and this set of common roots is explicitly described.

In our case the b-function is not just a polynomial in $s$, it is a piecewise polynomial and the roots depend on $\omega$.
Unfortunately, the ideal Ann$\big(f^\ell\big)$ is not as simple as in the case of the cusp, and for $\mu$-constant deformations we do not have explicitly a set of generators of Ann$\big(f^\ell\big)$.

Let us look at an example.

\begin{Exam}
The simplest example is $f_0=x^4+y^5$, which admits only one $\mu$-constant deformation
\[f=x^4+y^5+x^2y^3\]
We compute some b-functions of Ann$\big(\frac{1}{f}\big)$ with Singular \cite{singular} and compare them. We denote the b-functions by $b_{0,\omega}(s)$ the one associated with $f_0$ and $b_{1,\omega}(s)$ the one associated with $f$. We can check that there are $\omega\in\R^2\setminus\{0\}$ for which the two b-functions coincide. Indeed, if $\omega=(5,4)$
\[b_{0,\omega}(s)=b_{1,\omega}(s)=s+20.\]
However, this is not the case whenever $p\omega_1=q\omega_2$, since, for $\omega=(-5,-4)$ we have the {\em opposite} behaviour
\[b_{0,\omega}(s)=s-20\]
\[b_{1,\omega}(s)=\big(s-9\big)\big(s-14\big)\big(s-\frac{62}{3}\big)\big(s-21\big)\big(s-\frac{64}{3}\big)\big(s-22\big)\]
the b-functions have no common roots. In fact, it seems that in the whole region  $$\{\lambda(1,1)+\mu(3,2)\ |\ \lambda,\mu\geq 0\}\subseteq\R^2$$ containing $(5,4)\R_{>0}$ we have equality of the b-functions, because the initial ideals are the same. While in the (open) region
\[\{\lambda(-1,-1)+\mu(-11,-8)\ |\ \lambda,\mu>0\}\]
the b-functions have no roots in common.

Finally, for any other $\omega$ the number of common roots is non-zero, but the b-functions are different.
\end{Exam}
\label{remMU}
\end{rem}

\end{document}